\documentclass[12pt,reqno]{amsart}
\usepackage[utf8]{inputenc}
\usepackage{a4wide} 
\usepackage{amsthm,mathrsfs,amssymb,amsmath,amsfonts,bbm}
\usepackage[hidelinks]{hyperref}
\usepackage{esint}
\usepackage{enumerate}
\usepackage{versions}

\usepackage{color}

\allowdisplaybreaks

\numberwithin{equation}{section}

\newcommand{\Hm}[1]{\leavevmode{\marginpar{\tiny%
$\hbox to 0mm{\hspace*{-0.5mm}$\leftarrow$\hss}%
\vcenter{\vrule depth 0.1mm height 0.1mm width \the\marginparwidth}%
\hbox to 0mm{\hss$\rightarrow$\hspace*{-0.5mm}}$\\\relax\raggedright
#1}}}

\newtheorem{theorem}{Theorem}[section]
\newtheorem*{theorem*}{Theorem}
\newtheorem{corollary}[theorem]{Corollary}
\newtheorem{lemma}[theorem]{Lemma}
\newtheorem{proposition}[theorem]{Proposition}
\theoremstyle{definition}
\newtheorem{definition}[theorem]{Definition}
\theoremstyle{remark}
\newtheorem{remark}[theorem]{Remark}

  \DeclareMathOperator {\supp} {supp\,}

  \newcommand {\CC}{\mathbb C}  
  
  \newcommand {\NN}{\mathbb N}

  \newcommand {\RR}{\mathbb R}

  \newcommand {\BOUNDED}{\mathscr B}

  \newcommand {\eins} {\mathbbm 1}

  \newcommand {\om}{\omega}

  \newcommand {\norm}[1] {\| #1 \|}  

  \newcommand {\bignorm}[1]{\bigl\| #1 \bigr\|}
  \newcommand {\Bignorm}[1]{\Bigl\| #1 \Bigr\|}

  \newcommand {\dsp} {\displaystyle}
  \newcommand {\sL} {\mathscr L}
  \newcommand{\MS}{{X}}  
  
  \newcommand{\DIM}{n}

\newcounter{aufzi}
\newenvironment{aufzi}{\begin{list}{ {\upshape(\alph{aufzi})}}{
        \usecounter{aufzi}
        \topsep1ex
        \parsep0cm
        \itemsep1ex
        \leftmargin0.8cm
        \labelwidth0.5cm
        \labelsep0.3cm
}}
{\end{list}}

\newcounter{aufziH}

\newenvironment{aufziH}{
\begin{list}{\upshape(H\arabic{aufziH})}{
    \usecounter{aufziH}
    \topsep1ex
    \parsep0cm
    \itemsep1ex
    \leftmargin0.8cm
    \labelwidth0.5cm
    \labelsep0.3cm
}}{
\end{list}
}

\newcounter{aufzii}
\newenvironment{aufzii}{\begin{list}{\hfill {\upshape 
(\roman{aufzii})}}{
        \usecounter{aufzii}
        \topsep1ex
        \parsep0cm
        \itemsep1ex
        \leftmargin0.8cm
        \labelwidth0.5cm
        \labelsep0.3cm
         \itemindent0cm
}}
{\end{list}}

\newcounter{aufziii}

\begin{document}

\title{ A weak type $(p,a)$ criterion for operators, and applications } 

\author{Bernhard~Haak and El Maati~Ouhabaz} \address{ Institut de Math\'ematiques de Bordeaux, Universit\'e de Bordeaux, UMR CNRS 5251, 351 Cours de la Lib\'eration 33405, Talence.  France}  
\email{Bernhard.Haak@math.u-bordeaux.fr}
\email{Elmaati.Ouhabaz@math.u-bordeaux.fr}

\dedicatory{Dedicated to Peer Christian Kunstmann  on the occasion of his 60th birthday}

\begin{abstract} {Let $(\MS, d, \mu)$ be a space of homogeneous type
    and $\Omega$ an open subset of $\MS$. Given a bounded operator
    $T: L^p(\Omega) \to L^q(\Omega)$ for some
    $1 \le p \le q < \infty$, we give a criterion for $T$ to be of weak
    type $(p_0, a)$ for $p_0$ and $a$ such that
    $\frac{1}{p_0} - \frac{1}{a} = \frac{1}{p}-\frac{1}{q}$.  These
    results are illustrated by several applications including
    estimates of weak type $(p_0, a)$ for Riesz potentials
    $\sL^{-\frac{\alpha}{2}}$ or for Riesz transform type operators
    $\nabla \Delta^{-\frac{\alpha}{2}}$ as well as $L^p-L^q$ boundedness of
    spectral multipliers $F(\sL)$ when the heat kernel of $\sL$
    satisfies a Gaussian upper bound or an off-diagonal bound. We also
    prove boundedness of these operators from the Hardy space
    $H^1_\sL$ associated with $\sL$ into $L^a(\MS)$. By duality this
    gives boundedness from $L^{a'}(\MS)$ into $\text{BMO}_\sL$.}
 \end{abstract}
 
\vspace{1cm}

\maketitle



 \vspace{1cm}

 \section{Introduction and main results}\label{sec:main}

 This article deals with extrapolation of operators acting between
 Banach space-valued $L^p$ spaces over a metric measure space $\MS$
 endowed with a doubling measure $\mu$. The expression ``doubling''
 refers to the fact that there is some constant $C>0$ for which the
 volume of the doubled ball satisfies
 \[
   \mu( B(x, 2r) ) \le C \, \mu(B(x,r)).
 \]
 Here $x \in \MS$ and $r>0$ are arbitrary. Spaces enjoying this
 property are called spaces of {\it homogeneous type} and play an
 important role in harmonic analysis due to a degree of generality
 that permit a large range of applications, including classical
 Euclidean settings, analysis on manifolds, analysis on graphs or even
 on fractals.  Operators involved in these settings are often singular
 integral operators.  Let $T$ be an operator acting from $L^p(\MS, \mu)$
 for some $p \ge 1$. We say that $T$ is given by a (singular) kernel
 $K_T(x,y)$ if $Tf(x) = \int_\MS K_T(x,y) f(y) \, d\mu(y)$ for all $f$
 with bounded support and for almost all $x$ which do not belong to
 the support of $f$. The function $K_T$ is locally integrable away from
 the diagonal $\{(x,x): \; x \in \MS \}$. One of the most important questions
 on singular integral operators is to have sufficient conditions on
 the kernel which allow the operator $T$ to be bounded on
 $L^p(\MS, \mu)$ for a given $p \in (1, \infty)$. This subject has
 been studied for decades and the so-called $T 1$ or $T b$ theorems
 apply if $K_T(x,y)$ is a Calder\'on-Zygmund kernel. A different
 classical problem is to start with $T$ which is already bounded, say
 on $L^2(\MS, \mu)$, and search for conditions which allow us to
 extrapolate $T$ as a bounded operator on $L^p(\MS, \mu)$ for some or
 all $p \in (1, \infty)\setminus \{2\}$. The well known {\it almost
   $L^1$} condition of H\"ormander
 \[
 \sup_{y,  y' \in \MS} \int_{d(x,y) \ge 2 d(y,y')}  | K_T(x,y) - {K}_{T}(x,y') |   \, d\mu(x) < \infty
 \]
 implies that $T$ is of weak type $(1,1)$ and hence bounded on
 $L^p(\MS, \mu)$ for all $p \in (1, 2)$.  Note however, that, in
 practice, one needs the kernel $K_T$ to be H\"older continuous in the
 second variable in order to check this condition. A more suitable
 condition for non-smooth kernels was introduced by Duong and McIntosh
 \cite{Duong-McIntosh}. It says that if $(A_r)_{r > 0}$ is an
 approximation of the identity with kernel that decays sufficiently fast
 (Gaussian bounds for instance) and if
 \[
 \sup_{y \in \MS, r >0} \; 
  \int_{d(x,y) \ge 2 r} | {K}_T(x,y) - {K}_{TA_r}(x,y) |   \, d\mu(x) < \infty
\]
then $T$ is also of weak type $(1,1)$. This criterion is also valid
if the underlying space is any nontrivial open subset $\Omega$ of
$\MS$. Blunck and Kunstmann \cite{Blunck-Kunstmann} provide a
condition for $T$ to be of weak type $(p_0, p_0)$ for $p_0 > 1$. We
also refer to subsequent improvements and reformulations by Auscher
\cite{Auscher:memoir} and ter Elst and Ouhabaz \cite{Elst-Ouhabaz}.
 
The primary aim of the present paper is to provide a sufficient
condition for an operator $T: L^p(\MS, \mu) \to L^q(\MS, \mu)$ with
$p \le q$ to be of weak type $(p_0, a)$ for $ p_0 \le a$. The case
$ p_0 =a > 1$ recovers the result from \cite{Blunck-Kunstmann} and the
case $p_0 = a = 1$ recovers the result in \cite{Duong-McIntosh}. See
Theorem~\ref{thm:main} and Corollary~\ref{coro:main} below.  Before we
state explicitly our extrapolation results we introduce some notation.
 
\bigskip

Let $(\MS, d, \mu)$ be a metric measure space and denote again by
\[
  B(x,r) = \{y \in \MS: \; d(x,y) < r \}
\]
the open ball of center $x \in \MS$ and radius $r > 0$. Its volume is
denoted by $V(x,r) = \mu( B(x,r) )$. For $j \ge 1$, the annulus
$B(x, (j+1)r) \setminus B(x,j r)$ is denoted by $C_j(x,r)$ and
$C_0(x, r) := B(x, r)$.  We suppose that $(\MS, d, \mu)$ is of
homogeneous type. From the property
\[
  V(x, 2r) \le C \, V(x,r) \quad \forall  x \in \MS, \ r >0
\]
it follows that there exist constants  $n > 0$ and  $C_n$ such that
\begin{equation}
  \label{eq:gonflement-lambda}
  V(x, \lambda r) \le C_n \, \lambda^n V(x,r) \quad \forall  x \in \MS, \ r >0 \ {\rm and}\ \lambda \ge 1.
\end{equation}
Note that the constant $n$ is not unique since
\eqref{eq:gonflement-lambda} holds for any $m>n$ if it holds for $n$.
The dependence of $C_n$ on $n$ keeps us from taking in infimum over
all such $n$. In the sequel we take some possible, but reasonably
small value of $n$ for which the foregoing volume property is
satisfied.

\medskip
The following result provides our weak type extrapolation principle. We first state it in the framework of Lebesgue spaces over the homogeneous space $\MS$. A subsequent corollary  extends  the result to the general setting of open subsets of $\MS$. 

\begin{theorem}\label{thm:main}
  Let $(\MS, d, \mu)$ be a metric measure space of homogeneous type,
  let $(E, \norm{\cdot}_E)$ and $(F, \norm{\cdot}_F)$ be Banach
  spaces and let $p_0, p, q, a \in [1, \infty)$ be such that $p_0 < p \le q$
  where $\frac{1}{p} - \frac{1}{q} = \frac{1}{p_0} -  \frac{1}{a}$. Let 
  $T : L^p(\MS, \mu; E) \to L^q(\MS, \mu; F)$ be a bounded linear
  operator and suppose that there exists a bounded linear operator $S : L^p(\MS, \mu; E) \to L^q(\MS, \mu; F)$
  and a family of linear operators
  $(A_r)_{r>0}$ on $L^p(\MS, \mu; E)$ such that
\begin{aufziH}
\item \label{thm:main-a} for some sequence $(\omega_j)$ of
  non-negative numbers satisfying
  \( \dsp \sum_{j\ge 1} j^n \omega_j < \infty \) and for each
  $f \in L^{p_0}(\Omega; E)$ with bounded support and each ball $B(x, r)$
  containing its support, we have the off-diagonal bound
  \begin{equation} \label{eq:At-estimate}
    \Bigl(\tfrac{1}{V(x,(j+1)r)} \int_{C_j(x, r)} \bignorm{ (A_r f)(y) }_E^p  \,d\mu(y)\Bigr)^{\frac{1}{p}} \le \frac{\om_j}{V(x,r)^{\frac{1}{p_0}}}
    \; \norm{ f }_{L^{p_0}(\MS, \mu;E)}.
\end{equation}
\item there exist  $\delta, W > 0$ such that
\begin{equation} \label{eq:Hoermander-a-condition}  
 \Bigl( \int_{\MS \setminus B(x, (1{+}\delta)r)} \bignorm{ (T - SA_r) f(y) }_F^a\,d\mu(y)\Bigr)^{\frac{1}{a}}  \quad \le \quad W  \,\norm{ f  }_{L^{p_0}(\MS, \mu;E)}
\end{equation}
for all $x\in \MS$, $r>0$ and
$f \in L^{p_0}(\MS, \mu; E)\cap L^\infty(\MS, \mu; E)$ supported in
$B(x, r)$.
\end{aufziH}
\medskip\noindent
Then $T: L^{p_0}(\MS, \mu;  E) \to L^{a, \infty}(\MS, \mu; F)$ is bounded.
\end{theorem}

\noindent This theorem, as well as the following Corollary will be
proved in section~\ref{sec2}.

\begin{remark}\label{remark1-1}
  The choice of annuli with radii
  $C_j(x,r) = B(x, (j+1)r) \setminus B(x,j r)$ is not unique. We could
  also take
  $\widetilde{C}_j(x,r) := B(x, r\, 2^{j+1}) \setminus B(x,r\, 2^j
  )$. In this case, the condition on $\omega_j$ in the theorem becomes
  \( \sum_{j} 2^{n j} \omega_j < \infty \). This latter condition
  is sometimes more flexible than the first one, especially when the
  kernel of the approximation identity $A_r$ does not have an
  exponential decay but a merely a polynomial one.
\end{remark}

The above theorem is also valid on any non-empty open subset $\Omega$
of $\MS$. Note that $(\Omega, d, \mu)$ is not necessarily  a space of
homogeneous type. In the next result, $V(x,r)$ denotes, as before, the
volume of the ball $B(x,r)$ of $\MS$ (and not that of $\Omega$).

\begin{corollary}\label{coro:main}
  Let $(\MS, d, \mu)$ be a space of homogeneous type, and
  $\Omega\not=\emptyset$ be an open subset of $\MS$.  Let
  $(E, \norm{\cdot}_E)$ and $(F, \norm{\cdot}_F)$ be Banach spaces and
  $p_0, p, q, a \in [1, \infty)$ be such that $p_0 < p \le q$ where
  $\frac{1}{p} - \frac{1}{q} = \frac{1}{p_0} - \frac{1}{a}$. Let 
  $T : L^p(\Omega, \mu; E) \to L^q(\Omega, \mu; F)$ be a 
  bounded linear operator and that there exists a bounded linear operator
   $S : L^p(\Omega, \mu; E) \to L^q(\Omega, \mu; F)$ and a family of linear
  operators $(A_r)_{r>0}$ on $L^p(\Omega, \mu; E)$ such that
\begin{aufziH}
\item \label{coro:main-a} for some sequence $(\omega_j)$ of
  non-negative numbers satisfying
  \( \dsp \sum_{j\ge 1} j^\DIM \omega_j < \infty \) and for each
  $f \in L^{p_0}(\Omega; E)$ with bounded support and each ball
  $B(x, r)$ containing its support, we have the off-diagonal bound
  \begin{equation} \label{eq:At-estimate2} \Bigl(\tfrac{1}{V(x,
      (j+1)r)} \int_{\Omega \cap C_j(x, r)} \bignorm{ (A_r f)(y) }_E^p
    \,d\mu(y)\Bigr)^{\frac{1}{p}} \le \frac{\om_j}{V(x,
      r)^{\frac{1}{p_0}}} \; \norm{ f }_{L^{p_0}(\Omega, \mu;E)}.
\end{equation}
\item\label{coro:main-b} there exist  $\delta, W > 0$ such that
  \begin{equation} \label{eq:Hoermander-a-condition2} \Bigl(
    \int_{\Omega \setminus B(x, (1{+}\delta)r)} \bignorm{ (T - SA_r)
      f(y) }_F^a\,d\mu(y)\Bigr)^{\frac{1}{a}} \quad \le \quad W
    \,\norm{ f }_{L^{p_0}(\Omega,\mu;E)}
\end{equation}
for all $x\in \Omega$, $r>0$ and
$f \in L^{p_0}(\Omega,\mu; E)\cap L^\infty(\Omega,\mu; E)$ supported in
$B(x, r)$.
\end{aufziH}
\medskip \noindent
Then $T: L^{p_0}(\Omega, \mu;  E) \to L^{a, \infty}(\Omega, \mu; F)$ is bounded.
\end{corollary}

\begin{remark}\label{rem1} In the case that $F = \RR$
  and $T$ maps into positive functions, linearity of $T$ is not
  important and it can be replaced by the sub-linearity property
  $T(f+g)\le c\, (Tf + Tg)$ for some constant $ c > 0$.
\end{remark}

  The special case that $p_0{=}a{=}1$ is of particular interest, and
  it recovers known results. We state some observations for this case.

  First, Theorem~\ref{thm:main} (or Corollary~\ref{coro:main} for
  domains) gives a weak type $(1,1)$ result in this case. Suppose, in
  addition to the hypotheses of Theorem~\ref{thm:main} that the
  operators $T$ and $SA_r$ are given by (singular) kernels
  $\vec{K}_T(x,y)$ and $\vec{K}_{SA_r}(x,y)$, that is, $\vec{K}_T$ and
  $\vec{K}_{SA_r}$ are strongly measurable functions on $\MS\times\MS$
  with values in $\BOUNDED(E; F)$ and locally integrable on
  $\MS\times\MS \setminus \{(x,x): \; x \in \MS\}$ ensuring that
  \begin{equation}\label{eq:2-2-sing}
    T f(x) := \int_\MS \vec{K}_T(x,y) f(y) \, d\mu(y)
  \end{equation} 
  is well deﬁned as an element of $F$ for all
  $f \in L^\infty(\MS, \mu ; E)$ that have a bounded support and for all
  $x \not\in \supp(f)$ (and similarly for $SA_r$ and
  $\vec{K}_{SA_r}(x,y)$). It is then easy to check that the integral
  condition \ref{coro:main-b} in the previous theorem or in the corollary is
  satisfied if
  \begin{equation}\label{eq:2-5}
  \sup_{y \in \MS, r >0} \;  
  \int_{d(x,y) \ge (1{+}\delta)r} \norm{ \vec{K}_T(x,y) - \vec{K}_{SA_r}(x,y) }_{{\mathcal L}(E,F)}  \, d\mu(x) < \infty. 
\end{equation} 
Therefore, \eqref{eq:2-5} together with the remaining hypothesis from
Theorem~\ref{thm:main} implies that $T$ is bounded from
$L^1(\MS, \mu; E)$ into $L^{1, \infty}(\MS, \mu; F)$. If $E = F = \CC$
and $T = S$, this is the result of \cite{Duong-McIntosh}. A version of
\cite{Duong-McIntosh} with $S \not= T$ appears first in
\cite{Elst-Ouhabaz}, where it was used to study spectral multiplier
type results for degenerate elliptic operators. In these comments,
$\MS$ can be replaced by any non-trivial open subset $\Omega$ of
$\MS$.

 Next, let $T : L^p(\MS, \mu; E) \to L^p(\MS, \mu; F)$ be a bounded
  operator which is given by a (singular) integral
  $\vec{K}_T(x,y)$. Suppose that this kernel satisfies the so-called
  {\it almost $L^1$} condition of Hörmander
\begin{equation}\label{eq:2-6}
  \sup_{y,  y' \in \MS} \int_{d(x,y) \ge (1{+}\delta) d(y,y')}  \norm{ \vec{K}_T(x,y) - \vec{K}_{T}(x,y') }_{{\mathcal L}(E,F)}  \, d\mu(x) < \infty.
  \end{equation}
  Arguing exactly as in \cite{Duong-McIntosh} one proves that
  \eqref{eq:2-6} implies \eqref{eq:2-5} with an appropriate family of operators $(A_r)_{r>0}$. See also
  Proposition~\ref{prop:prop2-1} below for a more general version.
  Therefore, if $T: L^p(\MS, \mu; E) \to L^p(\MS, \mu; F)$ is bounded
  for some fixed $p \in (1, \infty)$ and the kernel $\vec{K}_T(x,y)$
  satisfies \eqref{eq:2-6}, then $T$ is weak type $(1,1)$, i.e.,
  $T: L^1(\MS, \mu; E) \to L^{1, \infty}(\MS, \mu; F)$ is
  bounded. This recovers Theorem~1.1 in Grafakos, Liu and Yang
  \cite{Grafakos-Liu-Yang}.

We extend in the next result the above comments to the case of weak
type $(1, a)$ operators with $a \ge 1$. Let
$T : L^p(\MS, \mu; E) \to L^q(\MS, \mu; F)$ be a bounded operator
which is given by a kernel $\vec{K}_T(x,y)$ in the sense of
\eqref{eq:2-2-sing}.

\begin{proposition}\label{prop:prop2-1}
  Let $a \in [1, \infty)$ and $ \delta > 0$. Consider the following
  properties.
\begin{aufzi}
\item \label{prop2-1-1}(Hörmander condition)
\begin{equation}\label{eq:2-1-Ho}
  \sup_{y,  y' \in \MS} \int_{d(x,y) \ge (1{+}\delta) d(y,y')}  \norm{ \vec{K}_T(x,y) - \vec{K}_{T}(x,y') }_{{\mathcal L}(E,F)}^a  \, d\mu(x) < \infty.
  \end{equation}
\item \label{prop2-1-2} There exists a family $(A_r)_{r>0}$ of linear
  operators on $L^p(\MS, \mu; E)$ which satisfies the off-diagonal
  bound \eqref{eq:At-estimate} and such that
\[
  \sup_{y \in \MS, r >0} \; 
  \int_{d(x,y) \ge (1{+}\delta)r} \norm{ \vec{K}_T(x,y) - \vec{K}_{TA_r}(x,y) }_{{\mathcal L}(E,F)}^a  \, d\mu(x) < \infty. 
\] 
\item \label{prop2-1-3} There exists  a constant $ W > 0$ such that
\[
  \Bigl( \int_{\MS \setminus B(x, (1{+}\delta)r)} \bignorm{ (T - TA_r)    f(y) }_F^a\,d\mu(y) \Bigr)^{\frac{1}{a}} \quad \le \quad W \,\norm{
    f }_{L^{1}(\MS,\mu;E)}
\]
for all $x\in \MS$, $r>0$ and
$f \in L^{1}(\MS,\mu; E)\cap L^\infty(\MS,\mu; E)$ supported in
$B(x, r)$.
\end{aufzi}
Then
$\ref{prop2-1-1} \Rightarrow \ref{prop2-1-2} \Rightarrow
\ref{prop2-1-3}$. In particular, condition \eqref{eq:2-1-Ho} implies
that the operator $T$ is of weak type $(1, a)$.
\end{proposition} 
Related results to the  fact that  \eqref{eq:2-1-Ho} implies that $T$ is of weak type $(1,1)$  are given in  Theorems 2.1 and 2.2 of H\"ormander
\cite{Hormander60} for convolution operators in the
Euclidean setting.   A variant of these results for vector-valued
kernels can be found in Rozendaal and Veraar \cite[Proposition~5.2]{RozendaalVeraar:2018}.

\bigskip

Our criteria for operators of weak type $(p_0, a)$ can be applied in
several situations. We are particularly interested in the endpoint
$p_0{=}1$ for Riesz potentials, Riesz transform type operators and
spectral multipliers. Let $\sL$ be the generator of a bounded
holomorphic semigroup $(e^{-t\sL})$ on $L^2(\MS)$, or on $L^2(\Omega)$
where $\Omega$ is an open subset of
$\MS$. 
We suppose that the semigroup $e^{-t \sL}$ is given by a kernel
$p_t(x,y)$, the heat kernel of $\sL$, which satisfies a Gaussian upper
bound
 \[
   | p_t(x,y) | \le \frac{C}{V(x, t^{\frac{1}{m}})} \exp\big\{-\delta 
   \left(\frac{d(x, y)} {t^{\frac{1}{m}}} \right)^{\frac{m}{m-1}}
   \big\}
\]
for some positive  constants  $C, \delta$ and $m > 1$. Then we prove 
the following result.

\begin{theorem}[Theorem 3.1]\label{thm:1-1}
Suppose that $\sL$ satisfies the Sobolev inequality
\[
  \norm{ u}_{L^{\frac{2D}{D-m}}(\Omega)} \le c \, \norm{
    \sL^{\frac{1}{2}} u }_{L^2(\Omega)} \quad \forall\, u \in
  D(\sL^{\frac{1}{2}})
\]
for some $D > m$ and  $c > 0$. Then the Riesz potential
$\sL^{-\frac{\alpha}{2}}$ is bounded from $L^1(\Omega)$ into
$L^{a, \infty}(\Omega)$ for $a > 1$ such that
$1 - \frac{1}{a} = \frac{m\alpha }{2 D}$.
\end{theorem}

\noindent The Sobolev inequality follows from the Gaussian bound in the case that  
\[
  V(x, r) \ge c\, r^D \quad \forall\, x \in \MS, \, r > 0.
\]
Indeed, the heat kernel decay
\[
  | p_t(x,y) |  \le C' t^{-\frac{D}{m}} \quad \forall t > 0
\]
is equivalent to the Sobolev inequality, see e.g. Davies
\cite[Theorem~2.4.2]{Davies:book} (note that the sub-Markov
property is not needed). We also refer to Coulhon \cite{Coulhon90}.

\medskip

Theorem~\ref{thm:1-1} is stated for operators with the heat kernel
satisfying a Gaussian upper bound of order $m$. The same proof can
also be used when the heat kernel has only an appropriate polynomial
decay (rather than exponential in the Gaussian case). There are
also examples  of operators for which the Gaussian upper bound is
not valid but the corresponding semigroup satisfies off-diagonal
estimates
\[
     \bignorm{ \eins_{A} \, e^{-t \sL} \, \eins_{B} f }_{L^q(\Omega)} \; \le \;
     C \, t^{\frac{-n}{m}(\frac1p-\frac1q)} \, \exp\big\{-\delta \left(\frac{d(A, B)} {t^{\frac{1}{m}}} \right)^{\frac{m}{m-1}} \big\} \;
     \norm{ f }_{L^p(\Omega)}
\]
for $ p_0 \le p \le q \le p_0'$ with some $p_0 > 1$. In this case, we
obtain the boundedness of the Riesz potential
$\sL^{-\frac{\alpha}{2}}$ from $L^{p_0}(\Omega)$ into
$L^{a, \infty}(\Omega)$ for
$\frac{1}{p_0} - \frac{1}{a} = \frac{m\alpha }{2 D}$. This applies to
second order elliptic operators with complex coefficients, higher
order elliptic operators, and Schr\"odinger operators with inverse
square potentials. Such examples can be found in several articles, see
e.g.  Blunck and Kunstmann \cite{Blunck-Kunstmann}.

\medskip

An interesting consequence of Theorem~\ref{thm:1-1} is that for a
non-negative self-adjoint operator $\sL$, $1 < p \le 2 \le q < \infty$
and $F: (0, \infty) \to \CC$ which has an appropriate decay at
infinity, the operator $F(\sL)$ is bounded from $L^p(\MS)$ into
$L^q(\MS)$.
For the case of the Euclidean Laplacian Hörmander \cite{Hormander60}
established a result involving functions $F$ in a suitable weak
Lebesgue space. In Proposition~\ref{prop:spec} we formulate and prove
a similar statement in our broader setting, albeit under slightly more
restrictive conditions on $F$.

\medskip

Another application of Theorem~\ref{thm:main} leads to the result that
for Laplace-Beltrami operator $\Delta$ on a complete Riemannian
manifold $\MS$, and assuming Gaussian upper bounds, the Riesz
transform type operator $\nabla \Delta^{-\frac{\alpha}{2}}$ is bounded
from $L^1(\MS)$ into $L^{a, \infty}(\MS)$ for
$1 - \frac{1}{p} = \frac{\alpha -1}{D}$, see
Proposition~\ref{prop:3-1} below. If $\alpha = 1$ this is a known
result of Coulhon and Duong \cite{Coulhon-Duong} who proved that the
Riesz transform $\nabla \Delta^{-\frac{1}{2}}$ is of weak type
$(1,1)$. The case $\alpha > 1$ does not seem to follow the natural
composition
$\nabla \Delta^{-\frac{\alpha}{2}} = ( \nabla \Delta^{-\frac{1}{2}})\,
\Delta^{-\frac{\alpha-1}{2}}$ and Theorem~\ref{thm:1-1}.

\medskip

We continue our investigation on endpoint estimates but we wish now to
have operators taking values in $L^a(\MS)$ instead of
$L^{a, \infty}(\MS)$. One has then to start with a suitable subspace
of $L^1(\MS)$ and, not surprisingly, it turns out that the Hardy space
$H^1_\sL$ associated with $\sL$ is an appropriate space.  We prove in
Proposition~\ref{prop:fractional-power-H1} the boundedness of
$\sL^{-\frac{\alpha}{2}}$ from $H^1_\sL$ into $L^a(\MS)$. This can be
compared with a result of Taibleson and Weiss \cite[Theorem~4.1,
p.101]{TaiblesonWeiss} in the Euclidean setting stating that the Riesz
potential is bounded from $H^p(\RR^D)$ to $H^q(\RR^D)$ for all
$0<p<\infty$ and $\tfrac1p - \tfrac1q =
\tfrac{\alpha}{D}$. 
From the boundedness of $\sL^{-\frac{\alpha}{2}}$ from $H^1_\sL$ into
$L^a(\MS)$ we then infer endpoint results for
$\nabla \Delta^{-\frac{\alpha}{2}}$ in Corollary~\ref{coro:riesz-h1}
and for $F(\sL)$ in Corollary~\ref{coro:spec-h1}. Following Duong and
Yan \cite{DuongYan} for the identification of the dual of $H^1_\sL$ we
finally obtain boundedness of $F(\sL)$ from $L^p(\MS)$ into
$\text{BMO}_\sL$.

\noindent We summarize some of these results in the following theorem.
We refer to Sections~\ref{sec:application:riesz} and ~\ref{section:H1}
for proofs, additional results, and comments.

\begin{theorem}\label{thm:1-2} Suppose that $\sL$ is a non-negative
  self-adjoint operator whose heat kernel has a Gaussian upper bound
  of order $m = 2$. Suppose also that $\sL$ satisfies the Sobolev
  inequality
  \[
    \norm{ u}_{L^{\frac{2D}{D-2}}(\MS)}  \le c \, \norm{ \sL^{\frac{1}{2}} u }_{L^2(\MS)}
\]
for all $u \in D(\sL^{\frac{1}{2}})$ and some $D > 2$ and $c > 0$.  We
have the following assertions.
\begin{aufzi}
\item The Riesz potential $\sL^{-\frac{\alpha}{2}}$ is bounded from
  $H^1_\sL(\MS)$ into $L^a(\MS)$ for
  $1 - \frac{1}{a} = \frac{\alpha }{ D}$.
\item Let $1 < p \le 2 \le q < \infty$ and $r$ such that
  $\frac{1}{r} = \frac{1}{p} -\frac{1}{q}$. Let $F: (0, \infty)$ be
  such that $| F(\lambda) | \le C\, \lambda^{-\frac{D}{2 r}}$ for all
  $\lambda > 0$. Then $F(\sL)$ is bounded from $ L^p(\MS)$ to
  $L^q(\MS)$.
\item Let $q \ge 2$ and $F: (0, \infty)$ such that
  $| F(\lambda) | \le C\, \lambda^{-\frac{D}{2q'}}$ for all
  $\lambda > 0$. Then $F(\sL)$ is bounded from $H^1_\sL$ into
  $L^q(\MS)$.
\end{aufzi}
\end{theorem}

\section{Proofs of the extrapolation results}\label{sec2}

\begin{proof}[Proof of Corollary~\ref{coro:main}]
  Borrowing an argument from \cite{Duong-McIntosh} for weak type
  $(1,1)$ operators, we can infer the Corollary from
  Theorem~\ref{thm:main}.  Indeed, let $\Omega$ be a non-trivial
  open subset of $\MS$.  We then extend all the operators
  $T, S, A_t, SA_r$ by zero outside $\Omega$, that is we consider
  $\widetilde{T} = \eins_\Omega T \eins_\Omega$,
  $\widetilde{S} = \eins_\Omega S \eins_\Omega$ and so on. Here
  $\eins_\Omega$ denotes the indicator function of $\Omega$. Then
  $(\widetilde{A_r})_r$ satisfies \eqref{eq:At-estimate} on $\MS$
  since $(A_r)_r$ satisfies \eqref{eq:At-estimate2} on
  $\Omega$. Condition \eqref{eq:Hoermander-a-condition2} on $\Omega$
  implies \eqref{eq:Hoermander-a-condition} for the extended operators
  on $\MS$. Now by Theorem~\ref{thm:main}, 
  $\widetilde{T} : L^{p_0}(\MS, \mu; E) \to L^{a, \infty}(\MS, \mu;
  F)$ is bounded and hence
  $T: L^{p_0}(\Omega, \mu; E) \to L^{a, \infty}(\Omega, \mu; F)$ is
  bounded as well.
\end{proof}

\begin{proof}[Proof of Theorem~\ref{thm:main}]
  Let $\alpha > 0$ and
  $f \in L^1(\MS, \mu; E) \cap L^{p_0}(\MS, \mu; E)$ with bounded support.
  We have to prove that for some constant $C$ (independent of $\alpha$
  and $f$)
\begin{equation}\label{eq:2-1}
  \mu\left( \bigl\{ x \in \MS:\; \norm{ Tf(x) }_F  > \alpha \bigr\} \right)^{\tfrac{1}a}  \le \frac{C}{\alpha} \norm{ f }_{L^{p_0}(\MS; E)}.
\end{equation} 
Fix $\beta > 0$ and write the Calderón-Zygmund
decomposition\footnote{see for instance {\cite{Grafakos-Liu-Yang}} for
  vector-valued functions when $p_0 = 1$. The case $p_0 > 1$ can be
  treated in a similar way as in the scalar case in
  {\cite{Blunck-Kunstmann}} or {\cite[Theorem~4.3.1,
    Exercise~4.3.8]{Grafakos}}.}  $f= g + b$ with the following
properties:
\begin{aufzii}
\item\label{item:CZ-1} $\norm{ g(x) } \le C \, \beta$ for $\mu-$a.e. $x\in \MS$
\item\label{item:CZ-2} $\dsp b = \sum_{i} b_i$, each $b_i$ is supported in a ball
  $B(x_i, r_i)$ and
  $\norm{ b_i }_{L^{p_0}(\MS;E)} \le C \, \beta \, V(x_i, r_i)^{\tfrac{1}{p_0}}$.
\item\label{item:CZ-3}
  \( \dsp \Bigl(\sum_i V(x_i, r_i) \Bigr)^{\tfrac{1}{p_0}} \le
  \frac{C}{\beta} \norm{ f }_{L^{p_0}(\MS; E)}. \)
\item\label{item:CZ-4} Each $x \in \MS$ is contained in at most $N$ of
  the balls $B(x_i, r_i)$.
\end{aufzii}
The constants $C$ and $N$ are independent of $f$ and $\beta$. We shall
use this decomposition with the choice
$\beta = \alpha^{\tfrac{a}{p_0}}$.

\medskip\noindent We may assume without loss of generality that
$\norm{ f }_{L^{p_0}(\MS; E)} = 1$. Since
\begin{align*}
   \mu\left( \left\{ x \in \MS:\;  \norm{ Tf(x) }_F > \alpha \right\} \right)
 \le & \; \mu\left( \left\{ x \in \MS: \; \norm{ Tg(x) }_F > \tfrac{\alpha}{2} \right\} \right)\\
     & + \mu\left( \left\{ x \in \MS:\;  \norm{ Tb(x) }_F > \tfrac{\alpha}{2} \right\} \right)
\end{align*}
we estimate separately each term on the right hand side. We start with
the ``good'' part. Let us abbreviate for simplicity
$\norm{ h }_r = \norm{ h }_{L^{r}(\MS; E)}$ for
$h \in L^r(\MS, \mu; E)$.  The assumption that $T$ is bounded
from $L^{p}(\MS, \mu; E)$ to $L^{q}(\MS, \mu; F)$ with norm
$\norm{ T }_{p\to q}$ gives that
\begin{align*}
  \mu\left( \left\{ x \in \MS:\;  \norm{ Tg(x) }_F > \tfrac{\alpha}{2} \right\} \right)
 \le& \; \frac{2^q}{\alpha^q} \norm{ T g }_q^q\\
 \le& \; \frac{2^q}{\alpha^q} \norm{ T  }_{p\to q}^q \norm{ g }_p^q\\
 \le& \; \frac{2^q}{\alpha^q} \norm{ T  }_{p\to q}^q \norm{ g }_{p_0}^{\tfrac{ q p_0}{p}} \norm{ g }_{\infty}^{q(1-\tfrac{p_0}{p})}\\
 \le& \; \frac{C_1}{\alpha^q} \norm{ T  }_{p\to q}^q  \norm{ f }_{p_0}^{q \tfrac{p_0}{p}} \beta^{q(1-\tfrac{p_0}{p})},
\end{align*}
where we used \ref{item:CZ-1} from the Calderón-Zygmund decomposition
and the fact that
$\norm{ g }_{p_0} \le \norm{ f }_{p_0} + \norm{ b }_{p_0} \le C'
\norm{ f }_{p_0}$ (which, in turn, follows easily from \ref{item:CZ-2}
and \ref{item:CZ-3}). Recall that $\norm{ f}_{p_0} = 1$ and our choice
$\beta = \alpha^{\tfrac{a}{p_0}}$. The relation
$\frac{1}{p} - \frac{1}{q} = \frac{1}{p_0} - \frac{1}{a}$ allows us to
simplify the last term to
\( \frac{C_1}{\alpha^a}\norm{ T }_{p\to q}^q \), so that
\begin{equation}\label{eq:2-2}
   \mu\left( \left\{ x \in \MS:\;  \norm{ Tg(x) }_F > \tfrac{\alpha}{2} \right\} \right)^{\tfrac{1}{a}}  \le \frac{C_2}{\alpha} \norm{ T  }_{p\to q}^{\tfrac{q}{a}}.
\end{equation}

\medskip\noindent Next, we look at the ``bad'' part, and estimate
$\mu\left( \left\{ x \in \MS:\; \norm{ Tb(x) }_F > \tfrac{\alpha}{2}
  \right\} \right)$.  We further decompose
$Tb = \dsp\sum_i S A_{r_i} b_i + \sum_{i} (T- S A_{r_i}) b_i$ which
leads to
\begin{align*}
           \mu\left( \left\{ x \in \MS:\;  \norm{ Tb(x) }_F > \tfrac{\alpha}{2} \right\} \right)
  \le&  \; \mu\left( \left\{ x \in \MS:\;  \norm{ \sum_i S A_{r_i} b_i  }_F > \frac{\alpha}4 \right\} \right)\\
     &   + \mu\left( \left\{ x \in \MS:\;  \norm{ \sum_{i} (T- S A_{r_i}) b_i }_F > \frac{\alpha}4 \right\} \right).
\end{align*}
We start by estimating $\dsp \norm{ \sum_i A_{r_i} b_i }_E$. To this
end, let $u$ be in the dual space $L^{p'}(\MS, \mu; E')$ and denote
the duality $L^p-L^{p'}$ by $\langle , \rangle_{p,p'}$. Then for
fixed $i \in \NN$, using \eqref{eq:At-estimate}
\begin{align*}
       | \langle A_{r_i} b_i,  u \rangle_{p,p'} |
  \le & \; \sum_j \int_{C_j(x_i,r_i)} \norm{ A_{r_i} b_i (y)}_E \norm{ u(y) }_{E'} \, d\mu(y)\\
\le   & \; \sum_j \left( \int_{C_j(x_i,r_i)} \norm{ A_{r_i} b_i(y) }_E^p \, d\mu(y) \right)^{\tfrac{1}p} \left( \int_{C_j(x_i,r_i)} \norm{ u(y) }_{E'}^{p'} \, d\mu(y) \right)^{\tfrac{1}{p'}}\\
\le   & \; \sum_j \omega_j \frac{V(x_i, (j+1)r_i)^{\tfrac{1}p}}{ V(x_i, r_i)^{\tfrac{1}{p_0}}}  \norm{ b_i }_{p_0} \left( \int_{C_j(x_i,r_i)} \norm{ u(y) }_{E'}^{p'} \, d\mu(y) \right)^{\tfrac{1}{p'}}
\end{align*}
where we used that $b_i$ is supported in $B(x_i, r_i)$. Now we use the
doubling property and \ref{item:CZ-2} from the Calderón-Zygmund
decomposition, which leads, up to some inessential constants
$C_1, C_2, \ldots$ to
\begin{align*}
       | \langle A_{r_i} b_i, u \rangle_{p,p'}|
  \le & \; C_1 \beta  \sum_j \omega_j (j+1)^{\tfrac{\DIM}p} V(x_i, r_i)^{\tfrac{1}{p}} \left( \int_{C_j(x_i,r_i)} \norm{ u(y) }_{E'}^{p'} \, d\mu(y)  \right)^{\tfrac{1}{p'}}\\
   =  & \; C_1 \beta  \sum_j \omega_j (j+1)^{\tfrac{\DIM}p} V(x_i, r_i)^{\tfrac{1}{p}} \, V(x_i, (j+1) r_i)^{\tfrac{1}{p'}} \\
      & \times  \left( \frac{1}{V(x_i, (j+1) r_i)} \int_{C_j(x_i,r_i)} \norm{ u(y) }_{E'}^{p'} \, d\mu(y)  \right)^{\tfrac{1}{p'}}\\
   \le& \; C_2 \beta  \sum_j \omega_j (j+1)^{\DIM} V(x_i, r_i) \left( {\mathcal M}(\norm{ u }_{E'}^{p'}) (z_i) \right)^{\tfrac{1}{p'}},
\end{align*}
where ${\mathcal M}$ is the uncentered Hardy-Littlewood maximal
operator and $z_i \in B(x_i, r_i)$ is arbitrary.  Using the assumption
on $(\omega_j)$ and averaging over $z_i \in B(x_i, r_i)$ yields
\[
  | \langle A_{r_i} b_i, u \rangle_{p,p'}| \le C_3 \beta \int_{B(x_i,    r_i)} \left( {\mathcal M}(\norm{ u }_{E'}^{p'}) (z_i) \right)^{\tfrac{1}{p'}}  \, d\mu(z_i).
\]
We use property \ref{item:CZ-3} and \ref{item:CZ-4} from the
Calderón-Zygmund decomposition and the fact that ${\mathcal M}$ is
of weak type $(1,1)$ to obtain
\begin{align*}
  | \langle \sum_i A_{r_i} b_i, u \rangle_{p,p'}|
  \le& \; C_4 \,\beta \int_{\bigcup_i B(x_i, r_i)} \left( {\mathcal M}(\norm{ u }_{E'}^{p'}) (z_i) \right)^{\tfrac{1}{p'}} \, d\mu(z) \\
  \le& \; C_5 \,\beta \left( \mu\left( \bigcup_i B(x_i,r_i) \right) \right)^{\tfrac{1}{p}} \Bignorm{ \left( {\mathcal M}(\norm{ u }_{E'}^{p'}) \right)^{\tfrac{1}{p'}} }_{L^{p', \infty}(\MS)}\\
   \qquad \le& \; C_6 \,\beta \frac{1}{\beta^{\tfrac{p_0}p}} \norm{ u }_{p'}\\
  =  & \;  C_6 \,\alpha^{\tfrac{a}p_0 (1-\tfrac{p_0}p)} \norm{ u }_{p'} 
  = \; C_6 \,\alpha^{a(\frac{1}{p_0} - \frac{1}{p} )} \norm{ u }_{p'}.
\end{align*}
This being true for each $u \in L^{p'}(\MS, \mu; E')$ with
constants that are independent of $u$, we conclude
\[
  \norm{  \sum_i A_{r_i} b_i }_p \le C_6 \alpha^{a(\frac{1}{p_0} - \frac{1}{p} )}.
\]
Using the  boundedness of  $S: L^p(\MS, \mu; E) \to L^q(\MS, \mu; F)$ we infer
\begin{align*}
  \mu\left( \left\{ x \in \MS:\;  \norm{ \sum_i S A_{r_i} b_i  }_F > \frac{\alpha}4 \right\} \right)
  \le& \; \frac{4^q}{\alpha^q} \norm{ S }_{p\to q}^q  \norm{  \sum_i A_{r_i} b_i }_p^q\\
  \le& \; \frac{C_7}{\alpha^a}  \norm{ S }_{p\to q}^q,
\end{align*}
that is
\begin{equation}\label{eq:2-3}
\mu\left( \left\{ x \in \MS:\;  \norm{ \sum_i S A_{r_i} b_i  }_F > \frac{\alpha}4 \right\} \right)^{\tfrac1a} \le \frac{C_7^{\tfrac1a}}{\alpha} \norm{ S }_{p\to q}^{\tfrac{q}{a}}.
\end{equation}

\medskip\noindent It remains to estimate
$\mu\left( \left\{ x \in \MS:\; \norm{ \sum_{i} (T- S A_{r_i}) b_i }_F >
    \frac{\alpha}{4} \right\} \right)$. At this stage we invoke the
hypothesis \ref{eq:Hoermander-a-condition} of the theorem. We have
\begin{align*}
    & \; \mu\left( \left\{ x \in \MS:\; \norm{ \sum_{i} (T- S A_{r_i}) b_i }_F > \frac{\alpha}{4} \right\} \right) \\
\le & \; \mu\left( \bigcup_i B(x_i, (1{+}\delta)r_i) \right) \\
    &  + \mu\left( \left\{ x \in \MS\setminus \bigcup_i B(x_i, (1{+}\delta)r_i):  \quad \norm{ \sum_{i} (T- S A_{r_i}) b_i }_F > \frac{\alpha}{4} \right\} \right).
\end{align*}
The first term on the right hand side is bounded by
$\dsp C(1{+}\delta)^\DIM \sum_i V(x_i,r_i)$, which in turn is bounded
by $\frac{C'}{\beta^{p_0}} = \frac{C'}{\alpha^{a}}$, using property
\ref{item:CZ-3} of the Calderón-Zygmund decomposition. For the second
term we have by the Kolmogorov inequality,
\begin{align*}
    & \;\mu\left( \left\{ x \in \MS\setminus \bigcup_i B(x_i, (1{+}\delta)r_i): \quad \bignorm{ \sum_{i} (T- S A_{r_i}) b_i }_F > \frac{\alpha}{4} \right\} \right)^{\tfrac1a}\\
\le & \; \frac{4}{\alpha} \left( \int_{\MS \setminus B(x_i, (1{+}\delta)r_i)} \Bignorm{ \sum_i (T - SA_{r_i}) b_i(y) }_F^a\,d\mu(y) \right)^{\tfrac1a}\\
\le & \; \frac{4}{\alpha} \sum_i  \left( \int_{\MS \setminus B(x_i, (1{+}\delta)r_i)} \bignorm{  (T - SA_{r_i}) b_i(y) }_F^a\,d\mu(y) \right)^{\tfrac1a}\\
\le & \; \frac{4 W}{\alpha} \sum_i \norm{ b_i }_{p_0}\\
\le & \; \frac{C_8 W}{\alpha}. 
\end{align*}
Note that we used \ref{item:CZ-2} and \ref{item:CZ-3} from the
Calderón-Zygmund decomposition in the last inequality.  These
estimates lead to the final estimate
\begin{equation}\label{eq:2-4}
\mu\left( \left\{ x \in \MS:\; \norm{ \sum_{i} (T- S A_{r_i}) b_i }_F > \frac{\alpha}{4} \right\} \right)^{\tfrac1a} \le \frac{C' + C_8 W}{\alpha}.
\end{equation}
Putting together \eqref{eq:2-2}, \eqref{eq:2-3} and \eqref{eq:2-4} we
obtain \eqref{eq:2-1} which proves the theorem.
\end{proof}

\begin{proof}[Proof of Proposition~\ref{prop:prop2-1}]
  The statement that condition \eqref{eq:2-1-Ho} implies that
  $T$ is of weak type $(1, a)$ follows from condition~\ref {prop2-1-3}
  and Theorem~\ref{thm:main}.

\smallskip\noindent  We prove that
  $\ref{prop2-1-1} \Rightarrow \ref{prop2-1-2}$ by a similar
  reasoning as in \cite{Duong-McIntosh}. Let
\[
J(x,y) = \norm{ \vec{K}_T(x,y) - \vec{K}_{TA_r}(x,y) }_{{\mathcal L}(E,F)}^a  \quad{\rm and} \quad I = \int_{d(x,y) \ge (1{+}\delta)r}   J(x,y) \, d\mu(x),
\]
where $A_r$ is the operator defined by a kernel
\[
  \vec{h_r}(z,y) = \frac{\eins_{B(y,r)}(z)}{V(y,r)} I_E
\]
in which $I_E$ denotes the identity operator on $E$. Clearly, this kernel has  a Gaussian upper bound and in particular \eqref{eq:At-estimate} holds for $p_0 = 1$. The kernel
$\vec{K}_{TA_r}(x,y)$ is then given by the usual composition
formula. Hence
\begin{align*}
J(x,y) = &  \; \norm{ \int_{d(z,y) \le r} [ \vec{K}_T(x,y) - \vec{K}_{T}(x,z)] \vec{h_r}(z,y)\, d\mu(z) }_{{\mathcal L}(E,F)}^a \\
\le& \; \left( \int_{d(z,y) \le r} \norm{ \vec{K}_T(x,y) - \vec{K}_{T}(x,z)}_{{\mathcal L}(E,F)} \norm{ \vec{h_r}(z,y) }_{{\mathcal L}(E,E)} \, d\mu(z) \right)^a\\
\le& \;  \int_{d(z,y) \le r} \norm{ \vec{K}_T(x,y) - \vec{K}_{T}(x,z)}_{{\mathcal L}(E,F)}^a \norm{ \vec{h_r}(z,y) }_{{\mathcal L}(E,E)}^a \, d\mu(z) V(y,r)^{a-1}.
\end{align*}
Using Fubini and  \ref{prop2-1-1}, we obtain for some constant
$C_1 > 0$, independent of $y, z$ and $r$, that
\begin{align*}
I \le& \; \int_{d(z,y) \le r} \int_{d(x,y) \ge (1+ \delta)d(z,y)}\norm{ \vec{K}_T(x,y) - \vec{K}_{T}(x,z)}_{{\mathcal L}(E,F)}^a\, d\mu(x)\,\\ 
     & \hspace*{2cm} \times  \norm{ \vec{h_r}(z,y) }_{{\mathcal L}(E,E)}^a \, d\mu(z) V(y,r)^{a-1}\\
  \le& \; C_1 \int_{d(z,y) \le r} \norm{ \vec{h_r}(z,y) }_{{\mathcal L}(E,E)}^a \, d\mu(z) V(y,r)^{a-1} = C_1,
\end{align*}
which is \ref{prop2-1-2}. 

\smallskip\noindent Assume now that that \ref{prop2-1-2} is satisfied. Let
$f \in L^{1}(\MS,\mu; E)\cap L^\infty(\MS,\mu; E)$ with support
contained in a ball $B(x, r)$. Then,
\begin{align*}
    &  \; \Bigl( \int_{\MS \setminus B(x, (1{+}\delta)r)} \bignorm{ (T - TA_r) f(y) }_F^a\,d\mu(y)\Bigr)^{\frac{1}{a}}\\
   =&  \; \Bignorm{ \int_\MS  ( \vec{K}_T(y,z) - \vec{K}_{TA_r}(y,z)) f(z)\, d\mu(z)}_{L^a(\MS \setminus B(x, (1{+}\delta)r), \mu ; F)}\\
\le &  \; \int_\MS  \bignorm{ ( \vec{K}_T(y,z) - \vec{K}_{TA_r}(y,z) ) f(z)\, d\mu(z)}_{L^a(\MS \setminus B(x, (1{+}\delta)r), \mu ;  F)}\\
\le&   \; \int_{\MS} \left( \int_{\MS \setminus B(x, (1{+}\delta)r)} \norm{\vec{K}_T(y,z) - \vec{K}_{TA_r}(y,z)}_{{\mathcal L}(E,F)}^a\, d\mu(y) \right)^{\tfrac{1}a} \norm{f(z)}_{F}\, d\mu(z)\\
\le&   \; C \int_\MS \norm{f(z)}_{F}\, d\mu(z) = C \norm{f}_{L^1(\MS, \mu;  F)}.
\end{align*}
This proves  property \ref{prop2-1-3}. 
\end{proof}

\section{Applications}\label{sec:application:riesz}
In this section we illustrate our main results by applications to
Riesz potentials, Riesz transform type operators and $L^p-L^q$ bounds
of spectral multipliers.

In the sequel we work for simplicity with Gaussian bounds, but we
mention that a polynomial decay of the heat kernel of high enough
order would suffice. 

\subsection{Riesz potentials}

Let $(\MS, \mu, d)$ be a space of homogeneous type and $\Omega$ a
non-trivial open subset of $\MS$.  Let $\sL$ be the generator of a
bounded holomorphic semigroup $(e^{-t\sL})$ on $L^2(\Omega,
\mu)$. Suppose that $e^{-t\sL}$ is given by a kernel $p_t(x,y)$, the
heat kernel of $\sL$, that is supposed to satisfy a Gaussian upper
bound of order $m > 1$,
\begin{equation}\label{eq:3-1}
  | p_t(x,y) | \le \frac{C}{V(x, t^{\frac{1}{m}})} \exp\big\{-\delta \left(\frac{d(x, y)} {t^{\frac{1}{m}}} \right)^{\frac{m}{m-1}} \big\} 
\end{equation}
for $x, y \in \Omega$ and $t > 0$. Here $C, \, \delta > 0$ are
constants. Using the doubling property we can replace
$V(x, t^{\frac{1}{m}})$ by $V(y, t^{\frac{1}{m}})$ at the expense of
changing the constant $\delta$.

Such Gaussian upper bounds are typical for elliptic operators of order
$m$ with $m \ge 2$. They are also satisfied for the Laplacian on some
fractals with a constant $m > 2$, called the walk dimension of the
fractal, see e.g. \cite{BarlowBass,Grigoryan-Hu-Lau}.

\begin{theorem}\label{thm:3-1}
  Suppose the Gaussian upper bound \eqref{eq:3-1}. 
Suppose that $\sL$ satisfies the Sobolev inequality 
\begin{equation}\label{eq:sob}
 \norm{ u}_{L^{\frac{2D}{D-m}}(\Omega)}  \le c \, \norm{ \sL^{\frac{1}{2}} u }_{L^2(\Omega)}
\end{equation}
for all $u \in D(\sL^{\frac{1}{2}})$ where $D > m$ and
$c > 0$ are constants. Let $\alpha > 0$. Then the Riesz potential
$\sL^{-\frac{\alpha}{2}}$ is bounded from $L^1(\Omega)$ into
$L^{a, \infty}(\Omega)$ for $a > 1$ that is defined by
$1 - \frac{1}{a} = \frac{m\alpha }{2 D}$. The Riesz potential is also
bounded from $L^p(\Omega)$ into $L^q(\Omega)$ for $1 < p  \le 2$ and 
$\frac{1}{p} - \frac{1}{q} = \frac{m \alpha }{2 D}$. If in addition the adjoint operator $\sL^*$ satisfies the same Sobolev inequality then $\sL^{-\frac{\alpha}{2}}$ is bounded from 
$L^p(\Omega)$ into $L^q(\Omega)$ for all  $1 < p < q < \infty$
with $\frac{1}{p} - \frac{1}{q} = \frac{m \alpha }{2 D}$.
\end{theorem}

\begin{remark}
\begin{aufzi} 
\item Note that the Sobolev inequality implies that $\sL^{\frac{1}{2}}$ is injective and so is the operator $\sL$ as well.
\item Suppose that $\MS = \RR^n$ endowed with the usual distance and
  Lebesgue measure.

  \smallskip\noindent Let either $\sL = - \text{div}(A(x) \nabla \cdot)$
  where the matrix $A$ has bounded real entries and is
  elliptic or let $\sL$ be the Schr\"odinger operator
  $\sL = \Delta + V$, were $\Delta$ is the non-negative Laplacian and
  $0 \le V \in L^1_{loc}(\RR^n)$. Then $\sL$ has a heat kernel which
  satisfies the Gaussian bound
 \[
    | p_t(x,y) | \le C t^{-\frac{n}{2}} \exp\big\{-\delta \frac{|x- y|^2} {t} \big\}.  
 \]
 Hence \eqref{eq:3-1} is satisfied with $m = 2$ and the Sobolev
 inequality \eqref{eq:sob} holds with $D = n$ (for $n > 2$).
 Consequently, $\frac{m\alpha }{2 D} = \frac{\alpha }{D}$ and so the
 theorem says that $\sL^{-\frac{\alpha}{2}}$ is bounded from
 $L^1(\RR^D)$ into $L^{a, \infty}(\RR^D)$ for
 $1 - \frac{1}{a} = \frac{\alpha }{D}$. The same statement is valid on
 any nontrivial open subset $\Omega$, when $\sL$ subject to Dirichlet
 boundary conditions. Our condition that
 $1 - \frac{1}{a} = \frac{m\alpha }{2 D}$ coincides then with the
 usual condition for Riesz potentials on $\RR^D$ or on domains of
 $\RR^D$.

 \smallskip\noindent Let $\sL$ be a higher order elliptic operator of
 order $m \in 2 \NN$ whose heat kernel satisfies
\[
 | p_t(x,y) | \le C t^{-\frac{n}{m}} \exp\big\{-\delta \left(\frac{|x-y|} {t^{\frac{1}{m}}} \right)^{\frac{m}{m-1}} \big\}.  
 \]
 Then the Riesz potential $\sL^{-\frac{\alpha}{2}}$ is bounded from
 $L^1(\Omega)$ into $L^{a, \infty}(\Omega)$ provided $a$ satisfies
 $1 - \frac{1}{a} = \frac{m\alpha }{2 D}$. The boundedness from
 $L^p(\Omega)$ into $L^q(\Omega)$ for $1< p < q < \infty$ with
 $\frac{1}{p} - \frac{1}{q} = \frac{m \alpha }{2 D}$ is obtained by
 the Marcinkiewicz interpolation theorem and it is consistent with the
 standard Sobolev embeddings.

\item Suppose that the volume $V(x,r)$ allows a polynomial lower bound
\begin{equation}\label{eq:3-1-vol}
  V(x, r) \ge c\,  r^D \quad \forall\, x \in \MS, \, r > 0.
\end{equation} 
It follows from the formula
\begin{equation}\label{eq:3-2-fr}
  \sL^{-\frac{\alpha}{2}} f = \tfrac1{\Gamma(\frac{\alpha}{2})} \int_0^\infty t^{\frac{\alpha}{2}-1} e^{-t \sL} f \, dt
\end{equation}
and the Gaussian bound that $\sL^{-\frac{\alpha}{2}}$ has a kernel
$k(x,y)$ which satisfies
\[
  | k(x,y) | \le C\, \int_0^\infty t^{\frac{\alpha}{2}-\frac{D}{m}}  \, \exp\big\{-\delta \left(\frac{d(x, y)} {t^{\frac{1}{m}}} \right)^{\frac{m}{m-1}} \big\} \frac{dt}t.  
\]
The change of variable $t = ( \frac{d(x,y)}{s})^m$ gives the estimate
\[ 
  | k(x,y)| \le \frac{C'}{ d(x,y)^{D - \frac{m \alpha}{2}}}.
\]
If, in addition, the volume has the polynomial growth
$V(x, r) \le C r^\beta$, then the conclusion of the theorem follows
from \cite{GarciaCuerva-Gatto}.  Observe that in
  Theorem~\ref{thm:3-1} we do not assume any upper or lower estimate
  for the volume. It would be interesting to find situations where
  \eqref{eq:sob} holds without such geometrical volume assumptions.
  Our statement reveals that the needed hypothesis is of a
  functional analytic nature more than a geometric one.

\end{aufzi}
\end{remark} 

\noindent Before we give the proof of Theorem~\ref{thm:3-1} we need the
following lemmata.

\begin{lemma}\label{lemma3-1}
  Let $ p \in (1, \infty)$. Under the assumptions of
  Theorem~\ref{thm:3-1}, there exist positive constants $C$ and
  $\delta'$ such that, for measurable subsets $A$ and $B$ of $\Omega$,
\[
\bignorm{\eins_{A} e^{-t\sL} \eins_B}_{{\mathcal L}(L^1(\Omega), \,L^p(\Omega))} \le 
C\,  t^{-\frac{D}{m}(1-\frac{1}{p})} \exp\big\{-\delta' \left(\frac{d(A, B)} {t^{\frac{1}{m}}} \right)^{\frac{m}{m-1}} \big\}.
\]
\end{lemma} 
\begin{proof} The operator $\eins_{A} e^{-t\sL} \eins_B$ is given by
  the kernel $K_t(x,y) = \eins_{A}(x) \eins_B(y) p_t(x,y)$. Hence
\begin{align*}
{}& \hspace{-1,5cm} \int_\Omega |K_t(x,y) |\, d\mu(x)\\
 \le& \; \frac{C}{V(y,t^\frac{1}{m})}\, \int_\Omega  \eins_{A}(x) \eins_B(y)\, \exp\big\{-\delta \left(\frac{d(x, y)} {t^{\frac{1}{m}}} \right)^{\frac{m}{m-1}} \big\}\, d\mu(x)\\
\le& \; \frac{C}{V(y,t^\frac{1}{m})}\, \exp\big\{-\frac{\delta}{2} \left(\frac{d(A, B)} {t^{\frac{1}{m}}} \right)^{\frac{m}{m-1}} \big\}
\int_\Omega \exp\big\{-\frac{\delta}{2}  \left(\frac{d(x, y)} {t^{\frac{1}{m}}} \right)^{\frac{m}{m-1}} \big\}\, d\mu(x)\\
\le& \; C'\, \exp\big\{-\frac{\delta}{2} \left(\frac{d(A, B)} {t^{\frac{1}{m}}} \right)^{\frac{m}{m-1}} \big\}.
\end{align*}
Note that we use here
\begin{equation}\label{eq:uniform}
\int_\Omega \exp\big\{-\frac{\delta}2 \left(\frac{d(x, y)} {t^{\frac{1}{m}}} \right)^{\frac{m}{m-1}} \big\}\, d\mu(x) \le C\, V(y,t^\frac{1}{m})
\end{equation} 
which follows easily by covering $\Omega$ with annuli $C(y, k)$ and using
the doubling property \eqref{eq:gonflement-lambda}.

\smallskip \noindent The above estimate for the $L^1$-norm of the
kernel $K_t(x,y)$ can be rephrased as
\begin{equation}\label{eq:norm1-1}
\norm{\eins_{A} e^{-t\sL} \eins_B}_{{\mathcal L}(L^1(\Omega))} \le 
C'\,   \exp\big\{-\frac{\delta}{2} \left(\frac{d(A, B)} {t^{\frac{1}{m}}} \right)^{\frac{m}{m-1}} \big\}.
\end{equation} 
On the other hand, the Sobolev inequality \eqref{eq:sob} and the fact
that the semigroup $e^{-t\sL}$ is uniformly bounded on $L^1(\Omega)$
and on $L^\infty(\Omega)$ (which both follow from \eqref{eq:uniform}),
the semigroup $e^{-t\sL} $ maps $L^1(\Omega)$ into $L^\infty(\Omega)$
with a norm that is controlled by $C''\, t^{-\frac{D}{m}}$, see
e.g. \cite[Theorem~2.4.2]{Davies:book} or \cite{Coulhon90}.
Therefore,
\begin{equation}\label{eq:norm1-inf}
\norm{\eins_{A} e^{-t\sL} \eins_B}_{{\mathcal L}(L^1(\Omega)), \,L^\infty(\Omega))} \le 
C''\,  t^{-\frac{D}{m}} \quad \forall\,  t > 0.
\end{equation} 
Now for $p \in (1, \infty)$ we use \eqref{eq:norm1-1},
\eqref{eq:norm1-inf} and interpolation to obtain the
lemma.
\end{proof}

\begin{lemma}\label{lemma3-2}
  Let $ p \in (1, \infty)$.  Under the assumptions of
  Theorem~\ref{thm:3-1} there exist positive constants $C$ and
  $\delta'$ such that, for every $f \in L^1(\Omega)$ supported in a
  ball $B(x, r)$,
\[
  \norm{e^{-t \sL} f }_{L^p(\Omega \setminus B(x, 2r))}
  \le
  C\,  t^{-\frac{D}{m}(1-\frac{1}{p})} e^{- \delta'   (\frac{r^m}{t})^{\frac{1}{m-1}}} \, \norm{ f }_{L^1(\Omega)}.
\]
\end{lemma} 
\begin{proof}
  Apply the previous lemma with $A = \Omega \setminus B(x, 2r))$ and
  $B = B(x,r)$.
\end{proof} 

\begin{proof}[Proof of Theorem~\ref{thm:3-1}]
  The assumed Sobolev inequality means that $\sL^{-\frac{1}{2}}$
  defines a bounded operator from $L^2(\Omega)$ into
  $ L^{\frac{2D}{D-m}}(\Omega)$. This implies that for $\alpha > 0$
  satisfying $\alpha < \frac{D}{m}$ and for $p = \frac{2D}{D-\alpha m}$, $\sL^{-\frac{\alpha}{2}}$ is
  bounded from $L^2(\Omega)$ to $L^p(\Omega)$, see for instance
  \cite{Coulhon90}. We rewrite the condition on $p$ as   
  $\frac{1}{2} - \frac{1}{p} = \frac{m \alpha}{2D}$. Now we have the starting point
  $\sL^{-\frac{\alpha}{2}}: L^2(\Omega) \to L^p(\Omega)$ for
  $\frac{1}{2} - \frac{1}{p} = \frac{m \alpha}{2D}$, it remains to
  check the two conditions of Theorem~\ref{thm:main} to obtain the
  endpoint
  $\sL^{-\frac{\alpha}{2}}: L^1(\Omega) \to L^{a, \infty}(\Omega)$. We
  choose $A_r = e^{-r^m \sL}$.

  \medskip \noindent First we prove \eqref{eq:At-estimate}. Let
  $x \in \MS$, $r > 0$, and $j \ge 0$. We define the operator
\[
  T_r = \eins_{C_j(x,r)} e^{-r^m \sL} \eins_{B(x,r)}.
\]
Then the operators $(T_r)_{r>0}$ are uniformly bounded on
$L^1(\Omega)$ since the semigroup is uniformly bounded on
$L^1(\Omega)$ by \eqref{eq:uniform}. On the other hand, the kernel of
$T_r$ is given by 
$K_r(z,y) = \eins_{C_j(x,r)}(z) p_{r^m}(z,y) \eins_{B(x,r)}(y)$ and satisfies
\begin{align*}
  | K_r(z,y) |
  \le & \; \eins_{C_j(x,r)}(z) \frac{C}{V(y,r)} \exp\big\{-\delta \left(\frac{d(z, y)} {r} \right)^{\frac{m}{m-1}} \big\} \eins_{B(x,r)}(y) \\  
  \le & \; \frac{C'}{V(x,r)} \exp\big\{-\delta j^{\frac{m}{m-1}} \big\}.   
\end{align*}
Since this bound is independent of $(z,y)$, it follows that
$T_r : L^1(\Omega) \to L^\infty(\Omega)$ with a norm that is
controlled by $\frac{C}{V(x,r)} \omega_j$ where
$\omega_j = \exp\big\{-\delta j^{\frac{m}{m-1}}
\big\}$.  
By complex interpolation, $T_r : L^1(\Omega) \to L^p(\Omega)$ for all
$p \in (1, \infty)$ and one obtains the first hypothesis
\ref{thm:main-a} of Theorem~\ref{thm:main}.

\medskip \noindent Next, we prove \eqref{eq:Hoermander-a-condition} with
$S = T = \sL^{-\frac{\alpha}{2}}$. 
By Lemma~\ref{lemma3-2} and definition of $a$, there exist positive constants $C$ and  $\delta'$ such that 
\begin{equation}\label{eq:3-3}
  \norm{e^{-s \sL}}_{\mathcal{L}(L^1(B(x,r)), \,L^a(\Omega\setminus B(x,2r)))} \le C\,  s^{-\frac{\alpha}{2}} \exp\big\{-\delta' \left(\frac{r^m} {s} \right)^{\frac{1}{m-1}} \big\}.
\end{equation}
We use \eqref{eq:3-2-fr} and recall our choice $A_r = e^{-r^m  \sL}$. Then
\begin{align*}
  (\sL^{-\frac{\alpha}{2}} - \sL^{-\frac{\alpha}{2}} e^{-r^m \sL}) f
  =& \; \tfrac1{\Gamma(\frac{\alpha}{2})} \int_0^\infty [ s^{\frac{\alpha}{2}-1} e^{-s \sL} f -  s^{\frac{\alpha}{2}-1} e^{-(s+r^m) \sL} f ]\, {ds}\\
  =& \; \tfrac1{\Gamma(\frac{\alpha}{2})} \int_0^\infty [ s^{\frac{\alpha}{2}-1} -  (s-r^m)^{\frac{\alpha}{2}-1} \eins_{\{ s > r^m \}} ] e^{-s \sL} f \, {ds}.
\end{align*}
Now by  \eqref{eq:3-3} we have for any $f$ with support
contained in a ball $B(x,r)$,
\begin{align*}
   & \; \Bigl( \int_{\Omega \setminus B(x, 2r)} |(\sL^{-\frac{\alpha}{2}} - \sL^{-\frac{\alpha}{2}} e^{-r^m \sL}) f(y)|^a\,d\mu(y)\Bigr)^{\frac{1}{a}}\\
\le& \; \frac{\norm{f}_{L^1(B(x,r))}}{\Gamma(\frac{\alpha}{2})}  \int_0^\infty \bigl| s^{\frac{\alpha}{2}-1} -  (s-r^m)^{\frac{\alpha}{2}-1} \eins_{\{ s > r^m \}}\bigr| \; \bignorm{e^{-s \sL}}_{\mathcal{L}(L^1(B(x,r)), \, L^a(\Omega\setminus B(x,2r)))}\, ds\\
\le& \; C  \norm{f}_{L^1(B(x,r))} \int_0^\infty \bigl|s^{\frac{\alpha}{2}-1} -  (s-r^m)^{\frac{\alpha}{2}-1} \eins_{\{ s > r^m \}} \bigr| s^{-\frac{\alpha}{2}} \exp\big\{-\delta' \left(\frac{r^m} {s} \right)^{\frac{1}{m-1}} \big\}\, ds. 
\end{align*}
By Lemma~\ref{lem:integral} below, this  integral expression is bounded by a
constant independent of $r$. This shows
\eqref{eq:Hoermander-a-condition} and finishes the proof the
$L^1-L^{a, \infty}$ estimate.

\medskip \noindent

Concerning $L^p-L^q$ boundedness  we can
either apply directly \cite{Coulhon90} or argue as follows.  By the Marcinkiewicz interpolation theorem
$\sL^{-\frac{\alpha}{2}}: L^p(\Omega) \to L^q(\Omega)$ for $1 < p \le 2 $ and 
$\frac{1}{p} - \frac{1}{q} = \frac{m \alpha }{2 D}$. For $p > 2$ we argue by duality. 
The heat kernel of $\sL^*$ obeys the same estimate as that of $\sL$,
therefore $\sL^*$ verifies the same $L^1-L^{a, \infty}$ bound.
We obtain as above that 
 $  (\sL^*)^{-\frac{\alpha}{2}}: L^p(\Omega) \to L^q(\Omega)$
for $1 < p \le 2 $ with
$\frac{1}{p} - \frac{1}{q} = \frac{m \alpha }{2 D}$. We take the adjoint and obtain the result. 
\end{proof} 

\noindent We state the following elementary lemma which already
appears in \cite{Coulhon-Duong}. We give a proof for the convenience
of the reader.

\begin{lemma}\label{lem:integral}
  Let $\delta, \gamma, \kappa, >0$. Then
  \[
    I_{\delta, \gamma, \kappa} := \sup_{t>0} \; \Big(\int_0^\infty | s^{\gamma-1} - (s-t)^{\gamma-1} \eins_{\{ s > t \}} | s^{-\gamma} e^{-\delta( t/s)^\kappa}\,ds\Bigr)  < \infty.
  \]
\end{lemma}
\begin{proof}
  The proof is straightforward. We cut the integral into the sum
  \[
     \int_0^t  s^{-1}  e^{-\delta (t/s)^\kappa}\,ds +  \int_t^\infty \bigl| s^{\gamma-1} - (s-t)^{\gamma-1} \bigr| s^{-\gamma} e^{-\delta (t/s)^\kappa}\,ds = I_1 + I_2.
\]
Observe that the $I_1$ coincides by the change of variables
$u = \tfrac{t}s$ with $\int_1^\infty e^{-\delta u^\kappa}\,\frac{du}u$
which is finite and independent of $t$. The second term $I_2$ is
translated to $(0,\infty)$, so that a subsequent change of variables
$s = t u$ yields
\[
  I_2 = \int_0^\infty \bigl| (1+u)^{\gamma-1} - u^{\gamma-1} \bigr|
  (1+u)^{-\gamma} e^{-\delta (\frac{1}{1+u})^\kappa}\,du.
\]
Convergence  close to zero is obvious for any $\gamma>0$. Hence
\begin{align*}
I_2 \le &  \; C + \int_1^\infty \bigl| (1+u)^{\gamma-1} - u^{\gamma-1} \bigr|
  (1+u)^{-\gamma} \,du\\
\le & \; C + |\gamma-1| \int_1^\infty \int_0^1  \frac{ (s+u)^{\gamma-2} }{(s+u)^\gamma} \,ds\,du =  C + |\gamma-1| \ln(2)
\end{align*}
using Fubini's theorem.  
\end{proof}

There are many situations where the semigroup $(e^{-t \sL})_{t\ge 0}$
does not enjoy a Gaussian upper bound. This is the case for example
for divergence form elliptic operators with bounded measurable and
complex coefficients or for higher order operators with non-smooth
coefficients. What is however true for these operators is an $L^p-L^q$
off-diagonal bound for $p, q$ in some interval around $2$. For these
operators we have a similar result to that found in
Theorem~\ref{thm:3-1}.

\begin{definition}
  Let $\Omega$ be a non-trivial open subset of $\MS$.
  We say that the semigroup $(e^{-t \sL})_{t\ge 0}$ admits an upper
   $L^p-L^q$ off-diagonal estimate of order $m > 1$, if there exists
   some $C, \delta >0$ such that
   \[
     \bignorm{ \eins_{A} \, e^{-t \sL} \, \eins_{B} f }_{L^q(\Omega)} \; \le \;
     C \, t^{\frac{-n}{m}(\frac1p-\frac1q)} \, \exp\big\{-\delta \left(\frac{d(A, B)} {t^{\frac{1}{m}}} \right)^{\frac{m}{m-1}} \big\} \;
     \norm{ f }_{L^p(\Omega)}
   \]
   for all measurable sets $A, B \subset \Omega$.
 \end{definition}

 In the following proposition we assume for simplicity that the volume
 in $\MS$ is polynomial, i.e. that \eqref{eq:ploynomial-vol} holds.

\begin{proposition}\label{prop:fractional-power-Lp}
 Suppose that there exists $c_1, c_2 > 0$ such that
 \begin{equation}   \label{eq:ploynomial-vol}
  c_1 \, r^D \le V(x, r) \le c_2 \, r^D \quad \forall x \in \MS, \, r > 0.
 \end{equation}
  Let $p_0 \in (1, 2)$ and suppose that for all $p \in (p_0, p_0')$,
  $(e^{-t \sL})_{t\ge 0}$ satisfies an $L^{p_0}-L^p$ off-diagonal
  bound of order $m$ for some $m < D$.  Let $\alpha > 0 $ such that
  $2_* := \frac{2D}{D-\alpha m} \le p_0'$ and suppose that
  $\sL^{-\frac{\alpha}{2}}: L^2(\Omega) \to L^{2_*}(\Omega)$ is
  bounded (Sobolev embedding). Let $a$ be defined by
  $\frac{1}{p_0} - \frac{1}{a} = \frac{m\alpha }{2 D}$.  Then
  $\sL^{-\frac{\alpha}{2}}$ is bounded from $L^{p_0}(\Omega)$ into
  $L^{a, \infty}(\Omega)$.
 \end{proposition}

The proof is a simple  adaptation of the proof of Theorem~\ref{thm:3-1}.

\subsection{Riesz transform type operators}
Our aim in this section is to prove $L^1-L^{a,\infty}$ estimates for
Riesz transform type operators
$\nabla \sL^{-\frac{1}{2}} \sL^{\frac{-\alpha}2}$ where $\sL$ has to
be some differential operator. The setting will be that $\MS$ is
either a complete Riemannian manifold and $\sL$ is the positive
Laplace-Beltrami operator $\Delta$ or that $\sL$ is a second order
elliptic operator in divergence form on a domain of $\RR^D$ with
Dirichlet boundary conditions. In the setting of a Riemannian
manifold, we assume the volume doubling property. Note that
$ \nabla \sL^{-\frac{1}{2}} \sL^{-\frac{\alpha}2}f(x)$ takes values in
the tangent space $T_x\MS$.

\begin{proposition}\label{prop:3-1} Suppose that the heat kernel
  $p_t(x,y)$ satisfies the Gaussian upper bound \eqref{eq:3-1} with
  $m = 2$. Suppose the Sobolev inequality \eqref{eq:sob} with $m= 2$
 and  some $D> 2$.  Let $\alpha > 1$ and let $a \in (1, 2]$ be such
  that $1 - \frac{1}{a} = \frac{\alpha - 1}{D}$. Then
  $\nabla \sL^{-\frac{\alpha}2}$ is bounded from $L^1(\MS)$ to
  $L^{a, \infty}(\MS, T\!\MS)$.
\end{proposition} 

\noindent A way to interpret this proposition is to say that if one
solves the elliptic problem $\sL^{\frac{\alpha}2} u = f$ for
$f \in L^1(\MS)$ in the sense that $u = \sL^{-\frac{\alpha}2} f$, then $\nabla\, u \in L^{a, \infty}(\MS)$.

\begin{proof} The arguments are exactly the same in the case of a
  manifold or Euclidean domain. So we consider the case of a manifold
  and $\sL = \Delta$.  We apply Theorem~\ref{thm:main} (or
  Corollary~\ref{coro:main} in the case of an Euclidean domain). We
  first need a starting point. For $u \in L^2(\MS)$ we have by
  Theorem~\ref{thm:3-1}
\begin{align*}
  \norm{\nabla \Delta^{-\frac{\alpha}2} u }_{L^2(\MS)}
   =& \; \norm{\Delta^{\frac{1}{2} -\frac{\alpha}2} u }_{L^2(\MS)}\\
   =& \; \norm{\Delta^{-\frac{\alpha -1}2} u }_{L^2(\MS)}\\
 \le& \; C \, \norm{u}_{L^p(\MS)}
\end{align*}
for $p$ such that $\frac{1}{p} - \frac{1}{2} = \frac{\alpha -
   1}{D}$.  Therefore, $\nabla \sL^{-\frac{\alpha}2}$ is bounded from $L^p(\MS)$ into
 $L^2(\MS, T\!\MS)$.

 \medskip\noindent Now we have to check the conditions of
 Theorem~\ref{thm:main}. We choose $A_r = e^{-r^2 \Delta}$ for which
 we have already checked the first hypothesis \ref{thm:main-a} in the
 proof of Theorem~\ref{thm:3-1}. It remains to check the second
 hypothesis \eqref{eq:Hoermander-a-condition} with
 $S = T = \nabla \Delta^{-\frac{\alpha}{2}}$. The big difference with
 the proof of Theorem~\ref{thm:3-1} comes from the presence of the
 gradient. When repeating the arguments we do not necessarily have
 pointwise bounds for the gradient of the heat kernel. Instead, we rely
 on the following weighted $L^2$ estimate from
 \cite{Grigoryan:JFA,Grigoryan:JDG} which is already used to study the
 Riesz transform in \cite{Coulhon-Duong},
\begin{equation}\label{eq:gri}
  \int_\MS | \nabla_y p_s(y,z) |^2 \, e^{\beta \frac{d(y,z)^2}{s}}\, d\mu(y) \le \frac{C}{s\, V(z, \sqrt{s})}
\end{equation}
for some constant $\beta > 0$ and all $s > 0$, $ z \in \MS$. Let $x \in \MS$, $r > 0$ and consider the operator
$T_s = \eins_{\MS\setminus B(x,2r)} \nabla e^{-s\Delta}
\eins_{B(x,r)}$. The kernel of $T_s$ is given by
$k_s(y,z) = \eins_{\MS\setminus B(x,2r)}(y) \nabla_y p_s(y,z)
\eins_{B(x,r)}(z)$. We estimate the $L^1$ norm of this kernel. So for
$\beta > 0$, satisfying \eqref{eq:gri}, 
\begin{align*}
  & \; \int_{\MS\setminus B(x,2r)} | \nabla_y p_s(y,z) | \eins_{B(x,r)}(z) \, d\mu(y)\\
= & \; \int_{\MS\setminus B(x,2r)} | \nabla_y p_s(y,z) | e^{\frac{\beta}{2} \frac{d(y,z)^2}{s}} e^{-\frac{\beta}{2} \frac{d(y,z)^2}{s}}\eins_{B(x,r)}(z) \, d\mu(y)\\
\le&\; e^{-\frac{\beta}{4} \frac{r^2}{s}} \left(\int_\MS | \nabla_y p_s(y,z) |^2 e^{\beta \frac{d(y,z)^2}{s}}\, d\mu(y) \right)^{\frac{1}{2}}\left( \int_\MS e^{-\frac{\beta}{2} \frac{d(y,z)^2}{s}}\, d\mu(y) \right)^{\frac{1}{2}}\\
\le&\; e^{-\frac{\beta}{4} \frac{r^2}{s}}\frac{C_1}{\sqrt{s\, V(z, \sqrt{s})}} \sqrt{V(z, \sqrt{s})} \\
=&\; e^{-\frac{\beta}{4} \frac{r^2}{s}}\frac{C_1}{\sqrt{s}}.
\end{align*}
Since the last term is independent of $z$ it follows that $T_s$ is a
bounded operator on $L^1(\MS)$ with norm controlled by
$\frac{C_1}{\sqrt{s}}e^{-\frac{\beta}{4} \frac{r^2}{s}}$. 
On the other hand, by analyticity of the semigroup $e^{-t \Delta}$ on $L^2(\MS)$
we have 
\begin{align*}
  \norm{ T_s f}_{L^2(\MS)}^2
  \le & \; \norm{ \nabla e^{-s \Delta} \eins_{B(x,r)} f}_{L^2(\MS)}^2\\
   =  & \; \int_\MS  \Delta e^{-s \Delta} \eins_{B(x,r)} f \cdot e^{-s \Delta} \eins_{B(x,r)} f \, d\mu\\
\le   & \; \norm{ \Delta  e^{-\frac{s}2 \Delta} e^{-\frac{s}2 \Delta} \eins_{B(x,r)} f }_{L^2(\MS)} \norm{  e^{-s \Delta} \eins_{B(x,r)} f }_{L^2(\MS)}  \\
\le   & \; \frac{C}{s}\, \norm{ e^{-\frac{s}2 \Delta} \eins_{B(x,r)}  f}_{L^2(\MS)}^2.
\end{align*}
Now the Sobolev inequality implies the $L^1-L^2$ estimate of
$e^{- \frac{s}2 \Delta}$ in terms of $C_1\, s^{-\frac{D}{4}}$.  Hence $T_s$ is
bounded from $L^1(\MS)$ into $L^2(\MS)$ with norm controlled by
$C_2\, s^{-\frac{D}{4} -\frac{1}{2}}$. Therefore, by complex
interpolation and using $1 - \frac{1}{a} = \frac{\alpha - 1}{D}$ we
have
\begin{equation}\label{eq:3-4} 
\norm{\eins_{\MS\setminus B(x,2r)} \nabla e^{-s\Delta} \eins_{B(x,r)}}_{\mathcal{L}(L^1(\MS), L^a(\MS))} \le C'\, e^{-\beta' \frac{r^2}{s}} s^{-\frac{\alpha}{2}}
\end{equation}
for some positive constants $C'$ and $\beta'$. Using \eqref{eq:3-2-fr}  this estimate implies  that for $f \in L^1(\MS)$ with support contained in $B(x,r)$ 
\begin{align*}
{}& \; \left(\int_{\MS \setminus B(x,2r)} | (\nabla \Delta^{-\frac{\alpha}{2}} - \nabla  \Delta^{-\frac{\alpha}{2}} e^{-r^2 \Delta})f(y)|^a\, d\mu(y) \right)^{\frac{1}{a}}\\
 =& \; \tfrac1{\Gamma(\frac{\alpha}{2})} \left(\int_{\MS \setminus B(x,2r)} \Bigl| \int_0^\infty [ s^{-\frac{\alpha}{2} -1} - (s-r^2)^{-\frac{\alpha}{2} -1} \eins_{\{s > r^2\}}] \nabla e^{-s \Delta} f(y)\Bigr|^a\, ds\,  d\mu(y) \right)^{\frac{1}{a}}\\
\le& \; C'\,  \Bigl(\int_0^\infty | s^{-\frac{\alpha}{2} -1} - (s-r^2)^{-\frac{\alpha}{2} -1} \eins_{\{s > r^2\}} |\, e^{-\beta' \frac{r^2}{s}} s^{-\frac{\alpha}{2}}\, ds\Bigr) \, \|f \|_{L^1(\MS)}\\
\le& \; C'' \, \|f \|_{L^1(\MS)}
\end{align*}
where we used again Lemma~\ref{lem:integral}. This proves condition
\eqref{eq:Hoermander-a-condition} and we appeal to Theorem~\ref{thm:main}
to conclude.
\end{proof}

\subsection{Spectral multipliers}

A well known result of Hörmander \cite{Hormander60} states that a
Fourier multiplier $T_F = {\mathcal F}^{-1} (F(\cdot) \mathcal{ F})$
is bounded from $L^p(\RR^D)$ to $L^q(\RR^D)$ provided
$1 < p \le 2 \le q < \infty$ and $F \in L^{r, \infty}(\RR^D)$ with
$\frac{1}{r} = \frac{1}{p} - \frac{1}{q}$. See also
\cite{RozendaalVeraar:2018} for a related result in the setting of
vector-valued Fourier multipliers. A sufficient 
condition to  have $F \in L^{r, \infty}(\RR^D)$ is  that
$|F(\xi)| \le C\, |\xi|^{-\frac{D}{r}}$ for all $\xi \in \RR^D \setminus \{0\}$.

\medskip\noindent A natural question is to ask whether a similar
result holds for more general operators than the Euclidean
Laplacian. More precisely, let $\sL$ be a non-negative self-adjoint
operator on $L^2(\MS)$, and $F : (0, \infty) \to \CC$ be a bounded
measurable function. Then $F(\sL)$ is bounded on $L^2(\MS)$. We wish
to have a condition close to Hörmander's  which implies that
$F(\sL)$ is bounded from $L^p(\MS)$ into $L^q(\MS)$. For $p = q$ there
are many results in this abstract setting, for instance in \cite{DOS}
where spectral multiplier results (i.e., $L^p$ to $L^p$) are proved
under the sole condition that the heat kernel of $\sL$ has a Gaussian
upper bound and $F$ satisfies some minimal regularity. In this
abstract setting we have

\begin{proposition}\label{prop:spec}
  Suppose the assumptions of Theorem~\ref{thm:3-1} and let $m > 1$ be the 
   constant  in the Gaussian upper bound \eqref{eq:3-1}.
   Let
  $1 < p \le 2 < q < \infty$ and let $r$ be such that
  $\frac{1}{r} = \frac{1}{p} - \frac{1}{q}$.  If the function
  $F: (0, \infty) \to \CC$ is such that
  $| F(\lambda) | \le C \, \lambda^{-\frac{D}{m r}}$ for all
  $\lambda > 0$, then $F(\sL): L^p(\MS) \to L^q(\MS)$ is bounded. 
\end{proposition}
Before we give the proof we compare this result with the
aforementioned result of Hörmander for Fourier multipliers. In the
case of the Laplacian, $m= 2$ so that our condition becomes
$| F(\lambda) | \le C\, \lambda^{-\frac{D}{2 r}}$. In our
setting the function is $G: \xi \mapsto F(|\xi|^2)$.  Thus our
condition reads $|G(\xi) | \le C \,  |\xi|^{-\frac{D}{ r}}$ which
implies  $G \in L^{r, \infty}(\RR^D)$.

\begin{proof}
  We write
  $F(\sL) = \sL^{-\frac{\alpha}{2}} \widetilde{F}(\sL)
  \sL^{-\frac{\beta}{2}}$ with
  $\widetilde{F}(\lambda) = F(\lambda) \lambda^{\frac{\alpha +
      \beta}{2}}$ for $\lambda > 0$ where $\alpha, \beta$ are positive
  constants which are chosen as follows. By Theorem~\ref{thm:3-1}
  \[
    \sL^{-\frac{\beta}{2}} : L^p(\MS) \to L^2(\MS)
    \quad\text{and}\quad
    \sL^{-\frac{\alpha}{2}} : L^2(\MS) \to L^q(\MS)
  \]
  provided that $\frac{1}{p} - \frac{1}{2} = \frac{m\beta}{2D}$ and
  $\frac{1}{2}- \frac{1}{q} = \frac{m\alpha}{2D}$. Thus,
  $F(\sL): L^p(\MS) \to L^q(\MS)$ is bounded as soon as
  $\widetilde{F}(\sL)$ is bounded on $L^2(\MS)$. This is the case if
  $\widetilde{F}$ is bounded on $(0, \infty)$, that is, if
  $| F(\lambda) | \le C\, \lambda^{-\frac{\alpha + \beta}{2}} =
  C \, \lambda^{-\frac{D}{m r}}$.
\end{proof}

\section{Boundedness from the Hardy space $H^1_\sL(\MS)$ into $L^{a}(\MS)$}\label{section:H1}

We have seen in the previous section examples of operators which are
bounded from $L^1(\MS)$ into $L^{a, \infty}(\MS)$. As in the classical
case of the Euclidean space, to ensure values in $L^a(\MS)$, one has
to restrict the operator to a subspace of $L^1(\MS)$. The convenient
choice for many problems is the Hardy space.

The classical Hardy space $H^1$ is well understood and a theory of
Hardy spaces $H^1_\sL$ associated with operators $\sL$ has been
developed in recent years, see e.g. \cite{DuongYan}.  Under
appropriate assumptions on $\sL$, $H^1_\sL$ coincides with the
classical Hardy space. This holds in particular when
$\sL = \Delta$ on $\RR^D$. In addition, the space $H^1_\sL$ satisfies
the usual interpolation property $[H^1_\sL, L^2]_\theta = L^p(\MS)$
for $\theta = \frac{2}{p} -1$. We refer to the specific memoir
\cite{HLMMY} on this subject, and the references therein.

\medskip

Let again  $(\MS, d, \mu)$ be a space of homogeneous type. Let $\sL$ be a non-negative self-adjoint operator in $L^2(\MS)$ and suppose that 
its semigroup $(e^{-t\sL})$ satisfies the $L^2$ Davies-Gaffney estimate. More precisely, there exist constants $C, c > 0$ such that for every open subsets $U_1, U_2 \subset \MS$ and every
$f_1, f_2 \in L^2(\MS)$ supported in $U_1$ and $U_2$
\[ | \langle e^{-t \sL} f_1, f_2 \rangle | \le C \exp\{- c \frac{ d(U_1, U_2)^2}{t}\} \norm{ f_1}_{L^2(\MS)} \norm{f_2}_{L^2(\MS)}.
\]
In our situation, we shall assume that the heat kernel of $\sL$ has a Gaussian upper bound of order $m= 2$. It is elementary to see that this implies the Davies-Gaffney estimate.

 Let $M \ge 1$. A function $b$ is called an $(M, \sL)$-atom if there exists
some ball $B(x,r)$ containing the support of $b$ and a function
$h \in L^2(\MS)$ such that
\begin{aufzii}
  \item\label{atom:i} $b = \sL^M h$
  \item\label{atom:ii} $\supp(\sL^k h) \subset B(x, r)$ for each $k = 0,\ldots,M$
  \item\label{atom:iii} $\bignorm{ (r^2 \sL)^k  h }_{L^2(\MS)} \le r^{2M} V(x, r)^{-\frac{1}{2}}$ for each $k = 0,\ldots,M$.
\end{aufzii}
A function $f = \sum_n \lambda_n b_n$ is an $(M, \sL)$-atomic decomposition if $(\lambda_n) \in \ell_1$, $(b_n)$ is a sequence of $(M, \sL)$-atoms and the sum is convergent in $L^2(\MS)$. For such function $f$ we define 
\[
  \norm{ f }_{H^1_\sL} = \inf \left\{ \sum_n |\lambda_n|: \quad
    f = \sum_n \lambda_n b_n \;\; \text{where}\; b_n \; \text{are} \;  (M, \sL) \text{--atoms} \right\}.
\]
The {\em atomic} Hardy space  $H^1_\sL$ is the completion with respect to this norm of the space of functions having an $(M, \sL)$-atomic decomposition. 
Since  $\sL$ satisfies Davies-Gaffney estimates, this space  coincides
(with equivalent norms) with a Hardy space defined via a square
function. We refer again to \cite{HLMMY}. We also mention that under our assumptions $H^1_\sL \subset L^1(\MS)$,  see \cite{AMM}. 

\medskip

\begin{proposition}\label{prop:fractional-power-H1}
  Suppose that
  the heat kernel $p_t(x,y)$ of  $\sL$
  satisfies the Gaussian upper bound
  \[
    | p_t(x,y) | \le \frac{C}{V(x, \sqrt{t})} \exp\big\{-\delta \frac{d(x, y)^2} {t} \big\} 
\]
for $x, y \in \MS$ and $t > 0$ where again $C, \, \delta > 0$ are
constants. Suppose also the Sobolev inequality \eqref{eq:sob} with
$m=2$.  Let $\alpha > 0$ and $a \ge 1$ such that
$1-\frac{1}{a} = \frac{\alpha}{D}$. Then the Riesz potential
$ \sL^{-\frac{\alpha}{2}}$ is bounded from $ H^1_{\sL}$ into
$L^a(\MS)$.
  \end{proposition} 
  
\begin{proof}
  It is enough to prove that for some constant $C > 0$
  \begin{equation}\label{eqH1-1}
  \norm{\sL^{-\frac{\alpha}{2}} b}_{L^a(\MS)} \le C
  \end{equation} 
  for every $(M, \sL)$-atom $b$. Indeed, taking an  $(M, \sL)$-atomic decomposition $f = \sum_n \lambda_n b_n$ with
  $\norm{ f }_{H^1_\sL} \approx   \sum_n |\lambda_n|$,  it then follows from \eqref{eqH1-1} that 
  \[  \sum_n | \lambda_n| \norm{\sL^{-\frac{\alpha}{2}} b_n}_{L^a(\MS)} \le C\, \sum_n | \lambda_n|  \le C'\, \norm{ f }_{H^1_\sL}.
  \]
  The series $\sum_n \lambda_n \sL^{-\frac{\alpha}{2}} b_n$ is
  absolutely convergent in $L^a(\MS)$ and hence it is convergent.  On
  the other hand, the series $f = \sum_n \lambda_n b_n$ is convergent
  in $L^2(\MS)$ and $ \sL^{-\frac{\alpha}{2}}$ is bounded from
  $L^2(\MS)$ to $L^q(\MS)$ with
  $\frac{1}{2} - \frac{1}{q} = \frac{\alpha}{D}$ (by \eqref{eq:sob},
  see Theorem~\ref{thm:3-1}) so that
  $\sL^{-\frac{\alpha}{2}} f = \sum_n \lambda_n
  \sL^{-\frac{\alpha}{2}} b_n$ as elements of $L^q(\MS)$.  Therefore,
  \[ 
  \norm{\sL^{-\frac{\alpha}{2}} f}_{L^a(\MS)} \le C'\, \norm{ f }_{H^1_\sL}
  \]
  and this estimate extends to all $f \in H^1_\sL$. 
  
  \medskip \noindent We now prove  \eqref{eqH1-1}. Let $b = \sL^M h$ be an
  $(M,\sL)$-atom where $M>\tfrac{\DIM}2$ satisfying \ref{atom:i}--~\ref{atom:iii}.  We decompose
  \begin{align*}
    \norm{ \sL^{-\frac{\alpha}{2}}  b }_{L^a(\MS)} \le
    & \; \norm{ \sL^{-\frac{\alpha}{2}}  b }_{L^a(B(x, 2r))} + 
    \norm{ \sL^{-\frac{\alpha}{2}} (I - e^{-r^2 \sL}) b }_{L^a( \MS \setminus B(x, 2r))} \\
    & \; +       \norm{ \sL^{-\frac{\alpha}{2}} e^{-r^2 \sL} b }_{L^a(\MS)}
  \end{align*}
  and estimate the three terms separately. 
  
  \medskip \noindent Step 1: we treat the first term. By H\"older's inequality
   \[
    \bignorm{ \sL^{-\frac{\alpha}{2}} b }_{L^a(B(x, 2r))} \le \bignorm{ \sL^{-\frac{\alpha}{2}} b }_{L^{2_*}(\MS)} V(x, 2r)^{\frac1a-\frac{1}{2_*}}.
  \]
  We arrange the value of $2_*$ here such that
  $\tfrac12 - \tfrac{1}{2_*} = 1 - \tfrac1a = \tfrac{\alpha}{D}$. Then  Theorem~\ref{thm:3-1} yields
  \[
    \bignorm{ \sL^{-\frac{\alpha}{2}} b }_{L^a(B(x, 2r))} \le C\,  \norm{b}_{L^2(\MS)} V(x, 2r)^{\frac{1}{2}} \le C'\,  \norm{b}_{L^2(\MS)} V(x, r)^{\frac{1}{2}}.
  \]
  Writing $b = \sL^M h$ by \ref{atom:i}, and using \ref{atom:iii} with 
  $k{=}M$ gives 
  \[
    \norm{ b }_{L^{2}(\MS)} = \norm{ \sL^M h }_{L^{2}(\MS)} \le C_1
    V(x, r)^{-\frac{1}{2}}.
  \]
  for some constant $C_1 > 0$ independent of $b$.  Hence
  \[
    \norm{ \sL^{-\frac{\alpha}{2}} b }_{L^a(B(x, 2r))} \le C_1.
  \]

  \medskip \noindent Step 2: we start with the representation
  \eqref{eq:3-2-fr}  that gives
  \begin{align*}
    & \;  \bignorm{ \sL^{-\frac{\alpha}{2}}  (I - e^{-r^2 \sL}) b }_{L^a( \MS \setminus B(x, 2r))} \\
    \le & \; \tfrac1{\Gamma(\frac{\alpha}{2})} \int_0^\infty \bigl|      s^{\frac{\alpha}{2}-1} - (s-r^2)^{\frac{\alpha}{2}-1}    \eins_{[\,s>r^2\,]} \bigr| \;  \bignorm{ e^{-s \sL} b }_{L^a( \MS  \setminus B(x, 2r))}\, ds.  \end{align*}
  Since $\supp(b) \subset B(x, r)$, we use  \eqref{eq:3-3} to obtain, as in  the  proof of Theorem~\ref{thm:3-1},   
  \[
    \bignorm{ \sL^{-\frac{\alpha}{2}} (I - e^{-r^2 \sL}) b }_{L^a( \MS \setminus B(x, 2r))}
    \le  C_2     \norm{b}_{L^1(\MS)} 
  \]
 since 
    \[
       \sup_{r>0}\; \int_0^\infty |   s^{\frac{\alpha}{2}-1} - (s-r^2)^{\frac{\alpha}{2}-1}    \eins_{[\,s>r^2\,]} | s^{-\frac{D}2(1-\frac1a)} e^{-\delta \frac{r^2}{s}} \,ds < \infty
    \]
  by observing that
  $\frac{D}2(1-\frac1a) = \frac{\alpha}2$ and appealing to
  Lemma~\ref{lem:integral}.  From this we deduce
   \[
     \bignorm{ \sL^{-\frac{\alpha}{2}} (I - e^{-r^2 \sL}) b }_{L^a( \MS \setminus B(x, 2r))}
     \le C_2 
      \norm{b}_{L^1(\MS)} 
     \le C_2
     \norm{b}_{L^2(\MS)} V(x,r)^{\frac{1}{2}}
     \le C_3
   \]
   by property \ref{atom:iii} for $k{=}M$.

  \medskip \noindent Step 3: for the last term, we use the atom property \ref{atom:i} of $b$ to write
  \[
      \bignorm{ \sL^{-\frac{\alpha}2} e^{-r^2 \sL} b }_{L^a(\MS)}
      = \bignorm{ \sL^{M-\frac{\alpha}2} e^{-\frac{r^2}2 \sL}  e^{-\frac{r^2}2 \sL}  h }_{L^a(\MS)}.
  \]
  The Sobolev inequality provides an  $L^1-L^a $ estimate 
  \[
    \norm{ e^{-t \sL}}_{L^1(\MS) \to L^a(\MS)} \le C_4\,  
    t^{-\frac{D}2(1-\frac1a)}.
  \]
  By the analyticity of the semigroup, we have with some inessential constants $C_5, C_6, \ldots$
  \begin{align*}
      \bignorm{ \sL^{-\frac{\alpha}2} e^{-r^2 \sL} b }_{L^a(\MS)}
      \le C_5 & \; \bigl(\frac{r^2}2\bigr)^{-(M-\frac{\alpha}2)} \bignorm{  e^{-\frac{r^2}2 \sL}  h }_{L^a(\MS)}\\
      \le C_6  & \; r^{\alpha-2M}  \bigl( \frac{r^2}2 \bigr)^{-\frac{D}2(1-\frac1a)} \norm{ h }_{L^1(\MS)}\\
      \le C_7 & \; r^{-2M}   \norm{ h }_{L^1(\MS)}. 
  \end{align*}
  Now the Cauchy-Schwarz inequality and the atom property \ref{atom:iii} for $k{=}0$
  allows to estimate further
  \[
    \norm{ h }_{L^1(\MS)} \le  \norm{ h }_{L^2(\MS)} V(x,r)^{\frac{1}{2}} \le  r^{2M},
  \]
  so that $\dsp  \bignorm{ \sL^{-\frac{\alpha}2} e^{-r^2 \sL} b }_{L^a(\MS)} \le  C_7$. \qedhere
\end{proof}

Identifying the dual space $( H^1_\sL)'$ with $\text{BMO}_\sL$ from  \cite{DuongYan}, we record

\begin{corollary}\label{cor:fractional-dual}
  Under the hypotheses of the previous proposition,
  $\sL^{-\frac{\alpha}2}: L^{\frac{D}\alpha}(\MS) \to \text{BMO}_\sL$
  is bounded for all $\alpha< D$.
\end{corollary}
\begin{proof}
  By  the previous proposition  
  $\sL^{-\frac{\alpha}2}: H^1_\sL \to L^a(\MS)$ is bounded for $1-\frac1a = \frac{\alpha}{D}$. The corollary follows by duality. 
\end{proof}
We mention that a related result to this corollary is proved in \cite{Dziubanski} for the particular case of $\sL = \Delta + V$ with some non-negative potential $V$. 
\begin{corollary}\label{coro:riesz-h1}
  Under the hypotheses Proposition~\ref{prop:3-1}, the Riesz transform type operator $\nabla \sL^{-\frac{\alpha}{2}}$ is bounded from $H^1_\sL$ into
  $L^a(\MS)$ for $a \le 2$ with $1 - \frac{1}{a} = \frac{\alpha-1}{D}$.
  \end{corollary}
  \begin{proof} 
  We write
  \[
  \nabla \sL^{-\frac{\alpha}{2}} = \nabla \sL^{-\frac{1}{2}}  \sL^{-\frac{\alpha-1}{2}}.
  \]
  By Proposition~\ref{prop:fractional-power-H1},
  $\sL^{-\frac{\alpha-1}{2}}: H^1_\sL \to L^a (\MS)$ is bounded. The
  Riesz transform $\nabla \sL^{-\frac{1}{2}}$ is bounded on $L^a(\MS)$
  by \cite{Coulhon-Duong} since we took  $a \le 2$.
  \end{proof}

  Finally, we have the following result for spectral multipliers. It
  is the endpoint result of Proposition~\ref{prop:spec}.
  
\begin{corollary}\label{coro:spec-h1}
  Suppose the assumptions of
  Proposition~\ref{prop:fractional-power-H1}. Let $q \ge 2$ and denote
  by $q'$ its conjugate. Let $F: (0, \infty) \to \CC$ be such that
  $| F(\lambda) | \le C\, \lambda^{-\frac{D}{2q'}}$ for all
  $\lambda > 0$.  Then $F(\sL)$ is bounded from $H^1_\sL$ to
  $L^q(\MS)$.
\end{corollary}
\begin{proof}
  As in the proof of Proposition~\ref{prop:spec} we write
  $F(\sL) = \sL^{-\frac{\alpha}{2}} \widetilde{F}(\sL)
  \sL^{-\frac{\beta}{2}}$ with
  $\widetilde{F}(\lambda) = F(\lambda) \lambda^{\frac{\alpha +
      \beta}{2}}$.  By Proposition~\ref{prop:fractional-power-H1},
  $\sL^{-\frac{\beta}{2}}$ is bounded from $H^1_\sL$ to $L^2(\MS)$
  provided $1 - \frac{1}{2} = \frac{\beta}{D}$, that is for
  $\beta = \frac{D}{2}$.  Next,  by Theorem~\ref{thm:3-1}, $\sL^{-\frac{\alpha}{2}}$ is bounded
  from $L^2(\MS)$ to $L^q(\MS)$ for
  $\frac{1}{2} - \frac{1}{q} = \frac{\alpha}{D}$. Now,
  $\widetilde{F}(\sL)$ is bounded on $L^2(\MS)$ if $\widetilde{F}$ is
  bounded on $(0, \infty)$. This later condition holds if
  $| F(\lambda) | \le C\, \lambda^{- \frac{\alpha + \beta}{2}} = C\,
  \lambda^{-\frac{D}{2q'}}$.
 \end{proof}

 We finish this section by some interesting observations on Schr\"odinger
 operators $\sL = \Delta + V$ on $\RR^D$.  Recall that in our
 notations, $\Delta$ is the non-negative Laplacian. We assume that
 $V $ is non-negative and belongs to the reverse H\"older class
 $RH_{\frac{D}{2}}$. We recall that $0 \le V \in RH_q$ if there exists
 a constant $C > 0$ such that
\[
  \left( \frac{1}{|B|} \int_B V^q\, dx \right)^{\frac{1}{q}}\le C\,  \left( \frac{1}{|B|} \int_B V\, dx \right)
\]
for all balls $B$ of $\RR^D$.

\begin{proposition}\label{prop:S-W.}
  Suppose that $D \ge 3$ and $0 \le V \in RH_{\frac{D}{2}}$. Then for
  $\alpha \in (0, D)$, there exists a positive constant $C$ such that
\[
  \norm{\sL^{-\frac{\alpha}{2}} f}_{L^{\frac{D}{D-\alpha}}} \le C\, \left[  \norm{f}_{L^1(\RR^D)} + \sum_{k=1}^D \norm{\frac{\partial}{\partial    x_k} \sL^{-\frac{1}{2}} f}_{L^1(\RR^D)} \right].
\]
\end{proposition} 
\begin{proof} By Proposition~\ref{prop:fractional-power-H1}
\[
  \norm{\sL^{-\frac{\alpha}{2}} f}_{L^{\frac{D}{D-\alpha}}} \le C\,  \norm{f}_{H^1_\sL}.
\]
By \cite[Theorem~4.1]{Duong-Ouhabaz-Yan} and
\cite[Lemma~6]{Dziubanski} we have
\[
  \norm{f}_{H^1_\sL} \le C'\, \bignorm{ \; \sup_{t > 0} \; \bigl| e^{-t\sL} f  \bigr| \; }_{L^1(\RR^D)}.
 \]
 By \cite[Theorem 1.7]{Dziubanski-Zienkiewicz} the norm
 $\norm{ \; \sup_{t > 0} | e^{-t\sL} f |\; }_{L^1(\RR^D)}$ is equivalent to
 $\norm{f}_{L^1(\RR^D)} + \sum\limits_{k=1}^D \norm{\frac{\partial}{\partial
     x_k} \sL^{-\frac{1}{2}} f}_{L^1(\RR^D)}$ and the result follows.
 \end{proof} 

 The case $V = 0$ in this proposition is well known.  It can for
 example be seen by combining the result
 \cite[Theorem~4.1,p.101]{TaiblesonWeiss} mentioned in the introduction
 with \cite[Corollary~1,p.~221]{Stein1970}.

\section*{Acknowledgement}
The authors would like to thank the anonymous referees for their
careful reading of the manuscript and for their insightful comments
and helpful suggestions, which have significantly improved the clarity
of this work.

\providecommand{\bysame}{\leavevmode\hbox to3em{\hrulefill}\thinspace}

\end{document}